\documentclass[11pt,reqno]{amsart}
\usepackage{amsfonts,amssymb,amsmath,color}
\usepackage{hyperref}

\newtheorem{theorem}{Theorem}

\newtheorem{proposition}[theorem]{Proposition}

\newtheorem{remark}[theorem]{Remark}

\numberwithin{equation}{section} \numberwithin{theorem}{section}

\begin{document}

\title{Pattern Recognition on Oriented Matroids: Symmetric Cycles in the Hypercube Graphs. IV}

\author{Andrey O. Matveev}
\email{andrey.o.matveev@gmail.com}

\begin{abstract}
We present statistics on the decompositions (with respect to a distinguished symmetric $2t$-cycle) of vertices of the hypercube graph,
whose negative parts are covered by two subsets of the ground set~$\{1,\ldots,t\}$ of the corresponding oriented matroid.
\end{abstract}

\maketitle

\pagestyle{myheadings}

\markboth{PATTERN RECOGNITION ON ORIENTED MATROIDS}{A.O.~MATVEEV}

\thispagestyle{empty}



\section{Introduction}

Let $t$ be an integer, $t\geq 3$, and let $\mathcal{M}:=(E_t,\mathcal{T})$ be a {\em simple\/} oriented matroid on the {\em ground set\/} $E_t:=[t]:=[1,t]:=\{1,\ldots,t\}$, with set of {\em topes\/}~$\mathcal{T}\subseteq\{1,-1\}^t$ that are regarded as row vectors of the real Euclidean space~$\mathbb{R}^t$. See~\cite{BLSWZ} on oriented matroids. Using a nonstandard terminology, by ``simple'' we mean that $\mathcal{M}$ has no loops, parallel or {\em antiparallel\/} elements.

We denote by $\mathrm{T}^{(+)}:=(1,\ldots,1)$ the $t$-dimensional row vector of all~$1$'s. If the oriented matroid $\mathcal{M}$ is {\em acyclic}, then $\mathrm{T}^{(+)}$ is called its {\em positive tope}. For a subset~\mbox{$A\subseteq E_t$,} we denote by~${}_{-A}\mathrm{T}^{(+)}$ the row vector $T\in\{1,-1\}^t$, whose {\em negative part\/} $T^-:=\{e\in E_t\colon T(e)=-1\}$ is the set~$A$.

Let $\boldsymbol{D}$ be a {\em symmetric cycle\/} in the {\em tope graph\/} of $\mathcal{M}$, that is, a $2t$-cycle with vertex sequence
$\mathrm{V}(\boldsymbol{D}):=(D^0,D^1,\ldots,D^{2t-1})$ such that
\begin{equation}
\label{eq:15}
D^{k+t}=-D^k\; ,\ \ \ 0\leq k\leq t-1\; .
\end{equation}
The subsequence of vertices $(D^0,\ldots,D^{t-1})$ is a {\em basis\/} of the space~$\mathbb{R}^t$. For any tope~$T$ of the oriented matroid~$\mathcal{M}$, there exists a unique row vector~\mbox{$\boldsymbol{x}:=\boldsymbol{x}(T):=\boldsymbol{x}(T,\boldsymbol{D}):=(x_1,\ldots,x_t)\in\{-1,0,1\}^t$} such that
\begin{equation*}
T=\sum_{i\in [t]}x_i\cdot D^{i-1}=\boldsymbol{x}\mathbf{M}\; ,
\end{equation*}
where
\begin{equation*}
\mathbf{M}:=\mathbf{M}(\boldsymbol{D}):=\left(\begin{smallmatrix}
D^0\\ D^1\vspace{-2mm}\\ \vdots \\ D^{t-1}
\end{smallmatrix}\right)\; .
\end{equation*}
Thus, the set
\begin{equation*}
\boldsymbol{Q}(T,\boldsymbol{D}):=\{x_i\cdot D^{i-1}\colon x_i\neq 0\}
\end{equation*}
is the {\em unique inclusion-minimal\/} subset (of {\em odd\/} cardinality) of
the vertex sequence~$\mathrm{V}(\boldsymbol{D})$ such that
\begin{equation}
\label{eq:16}
T=\sum_{Q\in\boldsymbol{Q}(T,\boldsymbol{D})}Q\; ,
\end{equation}
see~\cite[\S{}11.1]{M-PROM},\cite[\S{}1]{M-SC-II}; we have $\sum_{Q\in\boldsymbol{Q}(T,\boldsymbol{D})}\langle T,Q\rangle=t$,
where~$\langle\cdot,\cdot\rangle$ denotes the standard scalar product on~$\mathbb{R}^t$.

While the tope~${}_{-A}\mathrm{T}^{(+)}$ can be viewed as a kind of {\em characteristic vector\/} of a subset~\mbox{$A\subseteq E_t$}, the ``$\boldsymbol{x}$-{\em{}vector}'' ${}_{-A}\mathrm{T}^{(+)}\cdot\mathbf{M}^{-1}$ is yet another {\em linear algebraic\/} portrait of the set~$A$.

We first mention, in Section~\ref{sec:6}, (de)composition constructions in the discrete hypercube~$\{0,1\}^t$.
In~Section~\ref{sec:7} we discuss a few properties of (de)com\-positions~(\ref{eq:16}) of topes of the oriented matroid $\mathcal{M}$ with respect to arbitrary symmetric cycles~$\boldsymbol{D}$ defined by~(\ref{eq:15}). In Section~\ref{sec:8} we deal with the {\em hypercube graph\/}~$\boldsymbol{H}(t,2)$ of topes of the oriented matroid~$\mathcal{H}:=(E_t,\{1,-1\}^t)$ realizable as the {\em arrangement\/} of {\em coordinate hyperplanes\/} in~$\mathbb{R}^t$, see, e.g.,~\cite[Example~4.1.4]{BLSWZ}. We present statistics on the decompositions of vertices
$T:=(T(1),\ldots,T(t))$ of the {\em discrete hypercube\/} $\{1,-1\}^t$ with respect to a distinguished symmetric cycle~$\boldsymbol{R}:=(R^0,R^1,\ldots,R^{2t-1},R^0)$ of the graph~$\boldsymbol{H}(t,2)$, defined as follows:
\begin{equation}
\label{eq:17}
\begin{split}
R^0:\!&=\mathrm{T}^{(+)}\; ,\\
R^s:\!&={}_{-[s]}R^0\; ,\ \ \ 1\leq s\leq t-1\; ,
\end{split}
\end{equation}
and
\begin{equation}
\label{eq:18}
R^{k+t}:=-R^k\; ,\ \ \ 0\leq k\leq t-1\; .
\end{equation}

The smart {\em interval\/} organization of the {\em maximal positive basis\/}~$\mathrm{V}(\boldsymbol{R})$ of the space~$\mathbb{R}^t$
allows us to get the {\em linear algebraic\/} decompositions of vertices of~$\boldsymbol{H}(t,2)$ with respect to the cycle $\boldsymbol{R}$ in an {\em explicit} and {\em computation-free\/} way, see~\cite[Prop.~2.4]{M-SC-II}. Four assertions collected in~\cite[Prop.~2.4]{M-SC-II} also concern the {\em interval\/} structure of subsets of the ground set~$E_t$: any nonempty subset~$A\subseteq E_t$ is regarded as a disjoint union
\begin{equation*}
A=[i_1,j_1]\;\dot{\cup}\;[i_2,j_2]\;\dot{\cup}\;\cdots\;\dot{\cup}\;[i_{\varrho(A)},j_{\varrho(A)}]
\end{equation*}
of intervals of $E_t$, for some $\varrho(A)$, where
\begin{equation*}
j_1+2\leq i_2,\ \ \ j_2+2\leq i_3,\ \ \ldots,\ \ j_{\varrho(A)-1}+2\leq i_{\varrho(A)}\; .
\end{equation*}

We are interested in the decompositions~$\boldsymbol{Q}({}_{-A}\mathrm{T}^{(+)},\boldsymbol{R})$, $\boldsymbol{Q}({}_{-B}\mathrm{T}^{(+)},\boldsymbol{R})$ and $\boldsymbol{Q}(T,\boldsymbol{R})$ of topes~${}_{-A}\mathrm{T}^{(+)}$, ${}_{-B}\mathrm{T}^{(+)}$ and $T$ such that $T^-=A\cup B$.

We consider ordered two-member {\em Sperner families\/}~$(A,B)$ on the ground set $E_t$, but, of course, the case of {\em strict inclusion\/} of subsets~$F\subsetneqq G$ is implicitly taken into our consideration as well, if we set~$A:=F$ and~\mbox{$B:=G-F$}.

{\em Integer compositions\/} and {\em Smirnov words\/} (i.e., words in which adjacent letters always differ) over sufficiently large alphabets form a toolkit for a careful structural and enumerative analysis of single finite sequences of symbols and tuples of sequences. One could involve various {\em restricted integer compositions\/} in order to make such an analysis even more detailed. See, e.g.,~\cite{HM} on compositions and words, and~\cite{Prodinger} and~\cite[Appx]{M-SC-III} on Smirnov words.

A closely related family of research tools is that originating from ``the first\! {\em K\!aplansky's lemma}''~\cite[\S{}20]{K} that gives the number
\begin{equation}
\label{eq:19}
\tbinom{n-k+1}{k}
\end{equation}
of ways of selecting $k$ objects, {\em no two consecutive}, from $n$ objects arranged in a row, see~e.g.,~\cite[\S{}1.8]{C},\cite[\S{}2.3.15]{GJ},\cite{Ka},\cite[\S{}1.3]{Riordan},\cite[\S{}3.2]{Ryser} and~\cite{MS}; in other words,~(\ref{eq:19}) is the number of words of length~$n$, over the two-letter alphabet~$(\theta,\alpha)$, with~$k$ nonconsecutive letters~$\alpha$. Recall that the total number~$\sum_{k=0}^{\lfloor(n+1)/2\rfloor}\tbinom{n-k+1}{k}$ of such words of length~$n$ is the~{\em Fibonacci number\/}~$\mathsf{F}_{n+2}$, see~e.g.,~\cite[Eq.~(6.130)]{GKP},\cite[Ex.~1.35a]{Stanley}.

Together with the four assertions of~\cite[Prop.~2.4]{M-SC-II}, three players in our statistics on the decompositions of vertices of the discrete hypercube~$\{1,-1\}^t$, with respect to the distinguished symmetric cycle~$\boldsymbol{R}$ of the hypercube graph $\boldsymbol{H}(t,2)$, are as follows:

We denote by $\mathtt{c}(m;n)$ the number of {\em compositions\/} of a positive integer~$n$ with $m$ positive parts:
\begin{equation*}
\mathtt{c}(m;n):=\tbinom{n-1}{m-1}\; .
\end{equation*}

Let $(\theta,\alpha,\beta)$ be a {\em three-letter\/} alphabet. If $\mathfrak{s}',\mathfrak{s}''\in\{\theta,\alpha,\beta\}$, then we denote by
\begin{equation*}
\mathfrak{T}(\mathfrak{s}',\mathfrak{s}'';k,i,j)
\end{equation*}
the number of ternary Smirnov words, over the alphabet~$(\theta,\alpha,\beta)$ and with the {\em Parikh vector\/}~$(k,i,j)$, that start with the letter~$\mathfrak{s}'$ and end with the letter~$\mathfrak{s}''$; see~\cite[Rem.~4.2]{M-SC-III}. Similarly, we denote by
\begin{equation*}
\mathfrak{F}(\mathfrak{s}',\mathfrak{s}'';k,i,j,h)
\end{equation*}
the number of Smirnov words, over the {\em four-letter\/} alphabet $(\theta,\alpha,\beta,\gamma)$ and with the Parikh vector~$(k,i,j,h)$,
that start with a letter~$\mathfrak{s}'\in\{\theta,\alpha,\beta,\gamma\}$ and end with a letter~$\mathfrak{s}''\in\{\theta,\alpha,\beta,\gamma\}$; see~\cite[Rem.~4.3]{M-SC-III}.

\section{Coherent vertex (de)compositions in the discrete hypercubes $\{0,1\}^t$ and $\{1,-1\}^t$}
\label{sec:6}

Let $\widetilde{\boldsymbol{H}}(t,2)$ denote the {\em hypercube graph\/} on the vertex set~$\{0,1\}^t$; for vertices $\widetilde{T}'$ and $\widetilde{T}''$, the pair~$\{\widetilde{T}',\widetilde{T}''\}$ by definition is an {\em edge\/} of $\widetilde{\boldsymbol{H}}(t,2)$ if and only if the {\em Hamming distance\/} between the words $\widetilde{T}'$ and $\widetilde{T}''$ is $1$.

Following~\cite[\S{}1]{EI-I}, we recall that a family of interesting graphs related to the hypercube graph~$\widetilde{\boldsymbol{H}}(t,2)$ includes {\em Fibonacci cubes\/}~\cite{Hsu,HPL,K-Fib}, {\em Lucas cubes\/}~\cite{M-PC-ZC}, {\em generalized Fibonacci cubes\/}~\cite{IKR}, {\em $k$-Fibonacci cubes\/}~\cite{ESS}, {\em Fibonacci-run graphs}~\cite{EI-I, EI-II} and {\em daisy cubes\/}~\cite{KM}.

We define a {\em symmetric cycle\/} $\widetilde{\boldsymbol{D}}$ in the hypercube graph~$\widetilde{\boldsymbol{H}}(t,2)$ to be its~\mbox{$2t$-cycle}, with vertex sequence
$\mathrm{V}(\widetilde{\boldsymbol{D}}):=(\widetilde{D}^0,\widetilde{D}^1,\ldots,\widetilde{D}^{2t-1})$, such that
\begin{equation*}
\widetilde{D}^{k+t}:=\mathrm{T}^{(+)}-\widetilde{D}^k\; ,\ \ \ 0\leq k\leq t-1\; ,
\end{equation*}
cf.~(\ref{eq:15}).

{\em Coherent decompositions\/} of vertices in the hypercube graphs~$\widetilde{\boldsymbol{H}}(t,2)$ and $\boldsymbol{H}(t,2)$ are implemented by means of the standard conversions
\begin{align}
\nonumber
\{0,1\} &\to \{1,-1\}: & x &\mapsto 1-2x\; , & 0&\mapsto 1\; , & 1&\mapsto -1\; ,\\
\intertext{and}
\nonumber
\{1,-1\} &\to \{0,1\}: & z &\mapsto \tfrac{1}{2}(1-z)\; ,& 1&\mapsto 0\; , & -1&\mapsto 1\; ,
\end{align}
whose derived bijective maps are
\begin{align}
\nonumber
\{0,1\}^t &\to \{1,-1\}^t: & \widetilde{T} &\mapsto \mathrm{T}^{(+)}-2\widetilde{T}\; ,\\
\intertext{and}
\label{eq:4}
\{1,-1\}^t &\to \{0,1\}^t: & T &\mapsto \tfrac{1}{2}(\mathrm{T}^{(+)}-T)\; .
\end{align}

Given a symmetric cycle $\boldsymbol{D}$ in the hypercube graph $\boldsymbol{H}(t,2)$ on the vertex set~$\{1,-1\}^t$, let us regard bijection~(\ref{eq:4}) as the (de)composition
\begin{equation*}
\{1,-1\}^t \xrightarrow{\text{(\ref{eq:4})}} \{0,1\}^t:
\end{equation*}
\begin{equation*}
\begin{split}
T&=\sum_{Q\in\boldsymbol{Q}(T,\boldsymbol{D})}Q\\
\mapsto &\phantom{=}\;\,\tfrac{1}{2}\Bigl(\mathrm{T}^{(+)}-\sum_{Q\in\boldsymbol{Q}(T,\boldsymbol{D})}Q\Bigr)\\
&=\tfrac{1}{2}\Bigl(|\boldsymbol{Q}(T,\boldsymbol{D})|\mathrm{T}^{(+)}-\bigl(|\boldsymbol{Q}(T,\boldsymbol{D})|-1\bigr)\mathrm{T}^{(+)}
-\sum_{Q\in\boldsymbol{Q}(T,\boldsymbol{D})}Q\Bigr)\\
&=-\tfrac{1}{2}\bigl(|\boldsymbol{Q}(T,\boldsymbol{D})|-1\bigr)\mathrm{T}^{(+)}+
\sum_{Q\in\boldsymbol{Q}(T,\boldsymbol{D})}\tfrac{1}{2}\bigl(\mathrm{T}^{(+)}-Q\bigr)
\; .
\end{split}
\end{equation*}

\begin{remark}
For a symmetric cycle $\widetilde{\boldsymbol{D}}$ in the hypercube graph $\widetilde{\boldsymbol{H}}(t,2)$ on the vertex set $\{0,1\}^t$, and for any vertex $\widetilde{T}$ of $\widetilde{\boldsymbol{H}}(t,2)$, there exists
a {\em unique inclusion-minimal\/} subset $\widetilde{\boldsymbol{Q}}(\widetilde{T},\widetilde{\boldsymbol{D}})\subset\mathrm{V}(\widetilde{\boldsymbol{D}})$, of {\em odd} cardinality, such that
\begin{equation*}
\widetilde{T}=-\tfrac{1}{2}\bigl(|\widetilde{\boldsymbol{Q}}(\widetilde{T},\widetilde{\boldsymbol{D}})|
-1\bigr)\mathrm{T}^{(+)}
+\sum_{\substack{\widetilde{Q}\in\widetilde{\boldsymbol{Q}}(\widetilde{T},\widetilde{\boldsymbol{D}}):\\
\widetilde{Q}\neq(0,\ldots,0)}
}\widetilde{Q}\; .
\end{equation*}
\end{remark}

In practice, it is convenient to translate (de)composition problems, for the hypercube graph~$\widetilde{\boldsymbol{H}}(t,2)$ and its symmetric cycles, to the hypercube graph~$\boldsymbol{H}(t,2)$ on the vertex set~$\{1,-1\}^t$, and then to send their solutions (found more or less easily) back to the graph~$\widetilde{\boldsymbol{H}}(t,2)$.

Let $X,Y\in\{1,-1\}^t$ be two vertices of the hypercube graph~$\boldsymbol{H}(t,2)$, where~$t$ is {\em even}, and let $\widetilde{X},\widetilde{Y}\in\{0,1\}^t$ be the corresponding vertices of the hypercube graph~$\widetilde{\boldsymbol{H}}(t,2)$, namely,
$\widetilde{X}:=\tfrac{1}{2}(\mathrm{T}^{(+)}-X)$ and~\mbox{$\widetilde{Y}:=\tfrac{1}{2}(\mathrm{T}^{(+)}-Y)$}.
Let $\mathtt{hwt}(\widetilde{T}):=\langle\widetilde{T},\mathrm{T}^{(+)}\rangle$ denote the {\em Hamming weight\/} of a
vertex~\mbox{$\widetilde{T}\in\{0,1\}^t$}, that is, the number of $1$'s in $\widetilde{T}$.
We have
\begin{equation*}
\langle X,Y\rangle=0\ \ \Longleftrightarrow\ \ \langle\widetilde{X},\widetilde{Y}\rangle=\frac{2(\mathtt{hwt}(\widetilde{X})+\mathtt{hwt}(\widetilde{Y}))-t}{4}\; ,
\end{equation*}
and if $4|t$ (i.e., $t$ is divisible by $4$), then we have
\begin{multline*}
|X^-|=|Y^-|=:s\; ,\ \ \langle X,Y\rangle=0\\ \Longleftrightarrow\ \
\mathtt{hwt}(\widetilde{X})=\mathtt{hwt}(\widetilde{Y})=:s\; ,\ \
\langle\widetilde{X},\widetilde{Y}\rangle=s-\tfrac{t}{4}\; .
\end{multline*}

\section{A few properties of the decompositions of topes of oriented matroids}
\label{sec:7}

Let $\boldsymbol{D}$ be a symmetric cycle in the tope graph of a simple oriented matroid~$\mathcal{M}:=(E_t,\mathcal{T})$. For a tope $T\in\mathcal{T}$, we let $\mathfrak{q}(T):=\mathfrak{q}(T,\boldsymbol{D})$ denote the cardinality of the set $\boldsymbol{Q}(T,\boldsymbol{D})$:
\begin{equation*}
\mathfrak{q}(T):=|\boldsymbol{Q}(T,\boldsymbol{D})|\; ,
\end{equation*}
and if $T:={}_{-A}\mathrm{T}^{(+)}$ for some subset $A\subseteq E_t$, then we also write
$\mathfrak{q}(A):=\mathfrak{q}(A,\boldsymbol{D})$ instead of $\mathfrak{q}({}_{-A}\mathrm{T}^{(+)})$:
\begin{equation*}
\mathfrak{q}(A):=|\boldsymbol{Q}({}_{-A}\mathrm{T}^{(+)},\boldsymbol{D})|\; .
\end{equation*}

Recall that for the {\em graph distance\/} $d(T',T'')$ between topes $T',T''\in\mathcal{T}$ (i.e., the {\em Hamming distance\/} between the words $T'$ and $T''$), we have
\begin{equation*}
d(T',T'')=\tfrac{1}{2}\bigr(t-\langle T', T''\rangle\bigr)\; ,\ \ \ \langle T', T''\rangle=t-2d(T',T'')\; .
\end{equation*}
In particular, if $t$ is {\em even}, then we have
\begin{equation*}
\langle T', T''\rangle=0\ \ \Longleftrightarrow\ \ d(T',T'')=\tfrac{t}{2}\; .
\end{equation*}
Interesting subsets of vertices of the discrete hypercube~$\{1,-1\}^t$, with {\em zero\/} pairwise {\em scalar products}, are the rows of {\em Hadamard \mbox{matrices}}, see, e.g.,~\cite{Horadam,Seberry,S-Y}.

\subsection{Change of cycle, change of basis} $\quad$

If we are interested in the decompositions of a tope $T$ of a simple oriented matroid~$\mathcal{M}:=(E_t,\mathcal{T})$, with respect to
symmetric cycles~$\boldsymbol{D}'$ and~$\boldsymbol{D}''$ in the tope graph of $\mathcal{M}$, then changing cycles is essentially the same as changing from one (maximal positive) basis of the space~$\mathbb{R}^t$ to another. Therefore, we have
\begin{equation*}
\boldsymbol{x}(T,\boldsymbol{D}'')=\boldsymbol{x}(T,\boldsymbol{D}')\cdot\mathbf{M}(\boldsymbol{D}')\mathbf{M}(\boldsymbol{D}'')^{-1}\; .
\end{equation*}

\subsection{Decompositions, Inclusion--Exclusion and valuations} $\quad$

Let $\boldsymbol{R}$ be the distinguished symmetric cycle~(\ref{eq:17})(\ref{eq:18}) in the hypercube graph~$\boldsymbol{H}(t,2)$. Let~$\boldsymbol{\sigma}(1):=(1,0,\ldots,0)$ denote the first standard unit vector of the space~$\mathbb{R}^t$.

As noted in~\cite[Rem.~2.1]{M-SC-III}, for {\em disjoint\/} subsets~$A$ and~$B$ of~$E_t$, we have
\begin{equation*}
\boldsymbol{x}({}_{-A}\mathrm{T}^{(+)},\boldsymbol{R})+\boldsymbol{x}({}_{-B}\mathrm{T}^{(+)},\boldsymbol{R})
=\boldsymbol{\sigma}(1)+\boldsymbol{x}({}_{-(A\dot{\cup}B)}\mathrm{T}^{(+)},\boldsymbol{R})\; .
\end{equation*}
In fact, this is an {\em Inclusion--Exclusion\/} type relation, since
$\boldsymbol{\sigma}(1)=\boldsymbol{x}(\mathrm{T}^{(+)},\boldsymbol{R})$ $=\boldsymbol{x}({}_{-(A\cap B)}\mathrm{T}^{(+)},\boldsymbol{R})$.

Turning to arbitrary topes and symmetric cycles, we arrive at a general conclusion (see~\cite[p.~265]{Stanley} on the poset-theoretic context):
\begin{proposition}
\label{th:5}
Let~$\boldsymbol{D}$ be an arbitrary symmetric cycle~{\rm(\ref{eq:15})} in the hypercube graph~$\boldsymbol{H}(t,2)$.
\begin{itemize}
\item[\rm(i)] The map
\begin{equation*}
\mathbb{B}(t)\to\mathbb{Z}^t\; ,\ \ A\mapsto\boldsymbol{x}({}_{-A}\mathrm{T}^{(+)},\boldsymbol{D})\; ,
\end{equation*}
is a {\em valuation} on the {\em Boolean lattice}~$\mathbb{B}(t)$ of subsets of the set $E_t$, since for any two sets~\mbox{$A,B\in\mathbb{B}(t)$,} we have
\begin{equation*}
\boldsymbol{x}({}_{-A}\mathrm{T}^{(+)},\boldsymbol{D})
+\boldsymbol{x}({}_{-B}\mathrm{T}^{(+)},\boldsymbol{D})=
\boldsymbol{x}({}_{-(A\cap B)}\mathrm{T}^{(+)},\boldsymbol{D})+\boldsymbol{x}({}_{-(A\cup B)}\mathrm{T}^{(+)},\boldsymbol{D})\; .
\end{equation*}
As a consequence, we have
\begin{equation*}
\boldsymbol{x}({}_{-(A\cap B)}\mathrm{T}^{(+)},\boldsymbol{D})
+\boldsymbol{x}({}_{-(A\triangle B)}\mathrm{T}^{(+)},\boldsymbol{D})=
\boldsymbol{x}(\mathrm{T}^{(+)},\boldsymbol{D})+\boldsymbol{x}({}_{-(A\cup B)}\mathrm{T}^{(+)},\boldsymbol{D})\; .
\end{equation*}

\item[\rm(ii)] Let $\mathcal{A}:=\{A_1,\ldots,A_{\alpha}\}$ be a nonempty family of subsets of the set~$E_t$. Let
$\mu(\cdot,\cdot):=\mu_{P(\mathcal{A})}(\cdot,\cdot)$ denote the {\em M\"{o}bius function} of the poset $P(\mathcal{A})$ of all intersections $\bigcap_{s\in S}A_s$, $S\subseteq[\alpha]$, ordered by inclusion; the greatest element $\hat{1}:=\bigcup_{A\in\mathcal{A}}A$ of~$P(\mathcal{A})$ represents the empty intersection. We have
\begin{align*}
\boldsymbol{x}({}_{-\bigcup_{A\in\mathcal{A}}A}\mathrm{T}^{(+)},\boldsymbol{D})
&=-\sum_{S\subseteq[\alpha]\colon |S|>0}(-1)^{|S|}\cdot
\boldsymbol{x}({}_{-\bigcap_{s\in S}A_s}\mathrm{T}^{(+)},\boldsymbol{D})\\
&=-\sum_{B\in P(\mathcal{A})-\{\hat{1}\}}\mu(B,\hat{1})\cdot
\boldsymbol{x}({}_{-B}\mathrm{T}^{(+)},\boldsymbol{D})\; .
\end{align*}
\end{itemize}
\end{proposition}

\subsection{Circular translations of decompositions}
\begin{remark}
Let $T\in\{1,-1\}^t$ be a vertex of the hypercube graph~$\boldsymbol{H}(t,2)$, and let $\boldsymbol{D}$ be a symmetric cycle in~$\boldsymbol{H}(t,2)$. Suppose that
\begin{equation*}
(D^0,D^1,\ldots,D^{2t-1})=:\mathrm{V}(\boldsymbol{D})\supset\boldsymbol{Q}(T,\boldsymbol{D})=(D^{i_0},D^{i_1},\ldots,D^{i_{\mathfrak{q}(T)-1}})\; ,
\end{equation*}
for some indices $i_0<i_1<\cdots<\mathfrak{q}(T)-1$. For any $s\in\mathbb{Z}$, we have
\begin{equation*}
\sum_{0\leq j\leq \mathfrak{q}(T)-1}D^{(i_j+s)\!\!\!\!\mod{2t}}\in\{1,-1\}^t\; .
\end{equation*}
\end{remark}

\subsection{The negative parts of topes, graph distances, and the scalar products of topes} $\quad$

Let $\boldsymbol{D}$ be a symmetric cycle in the tope graph of a simple oriented matroid~$\mathcal{M}:=(E_t,\mathcal{T})$. According to~\cite[Rem.~13.2]{M-PROM}, for any tope~$T\in\mathcal{T}$ we have
\begin{align*}
\sum_{Q\in\boldsymbol{Q}(T,\boldsymbol{D})}d(T,Q)&=\tfrac{1}{2}\bigl(\mathfrak{q}(T)-1\bigr)t\\
\intertext{and}
\mathfrak{q}(T)&=1+\tfrac{2}{t}\sum_{Q\in\boldsymbol{Q}(T,\boldsymbol{D})}d(T,Q)\; .
\end{align*}
Since
\begin{equation*}
e\in E_t\; ,\ \ \ T(e)=-1\ \ \ \Longrightarrow\ \ \ |\{Q\in\boldsymbol{Q}(T,\boldsymbol{D})\colon Q(e)=-1\}|=\lceil\tfrac{\mathfrak{q}(T)}{2}\rceil\; ,
\end{equation*}
we see that
\begin{equation*}
\begin{split}
\sum_{Q\in\boldsymbol{Q}(T,\boldsymbol{D})}|Q^-|&=\sum_{e\in T^-}\lceil\tfrac{\mathfrak{q}(T)}{2}\rceil
+\sum_{e\in T^+}\lfloor\tfrac{\mathfrak{q}(T)}{2}\rfloor
\\&=|T^-|\cdot\lceil\tfrac{\mathfrak{q}(T)}{2}\rceil
+(t-|T^-|)\cdot\lfloor\tfrac{\mathfrak{q}(T)}{2}\rfloor\; ,
\end{split}
\end{equation*}
where~$T^+:=\{e\in E_t\colon T(e)=1\}$ is the {\em positive part\/} of the tope~$T$.

\begin{remark}
\label{th:1}
For any tope $T$ of a simple oriented matroid $\mathcal{M}:=(E_t,\mathcal{T})$, and for any symmetric cycle~$\boldsymbol{D}$ in the tope graph of~$\mathcal{M}$, we have
\begin{itemize}
\item[\rm(i)]
\begin{align}
\nonumber
\sum_{Q\in\boldsymbol{Q}(T,\boldsymbol{D})}|Q^-|&=|T^-|+\tfrac{1}{2}\bigl(\mathfrak{q}(T)-1\bigr)t\\
\intertext{and}
\label{eq:3}
\mathfrak{q}(T)&=1-\tfrac{2}{t}|T^-|+\tfrac{2}{t}\sum_{Q\in\boldsymbol{Q}(T,\boldsymbol{D})}|Q^-|\; .
\end{align}
\item[\rm(ii)]
\begin{equation}
\label{eq:2}
\sum_{Q\in\boldsymbol{Q}(T,\boldsymbol{D})}d(T,Q)=-|T^-|+\sum_{Q\in\boldsymbol{Q}(T,\boldsymbol{D})}|Q^-|\; .
\end{equation}
\end{itemize}
\end{remark}

As noted in~\cite[Rem.~13.3]{M-PROM}, if $T\not\in\mathrm{V}(\boldsymbol{D})$,
then we have
\begin{align}
\nonumber
\sum_{0\leq i< j\leq \mathfrak{q}(T)-1}d(Q^i,Q^j)&=\tfrac{1}{4}\bigl(\mathfrak{q}(T)^2-1\bigr)t\; ,\\
\nonumber
\mathfrak{q}(T)&=\sqrt{1+\tfrac{4}{t}\sum\nolimits_{i<j}d(Q^i,Q^j)}\; ,\\
\intertext{and}
\label{eq:1}
\sum_{0\leq i< j\leq \mathfrak{q}(T)-1}d(Q^i,Q^j)&=\tfrac{1}{2}\bigl(\mathfrak{q}(T)+1\bigr)\sum_{Q\in\boldsymbol{Q}(T,\boldsymbol{D})}d(T,Q)\; .
\end{align}

Relations (\ref{eq:2}) and (\ref{eq:1}) yield
\begin{multline*}
\sum_{0\leq i< j\leq \mathfrak{q}(T)-1}d(Q^i,Q^j)=\tfrac{1}{2}\bigl(\mathfrak{q}(T)+1\bigr)\sum_{Q\in\boldsymbol{Q}(T,\boldsymbol{D})}d(T,Q)
\\=
\tfrac{1}{2}\bigl(\mathfrak{q}(T)+1\bigr)
\Bigl(-|T^-|+\sum_{Q\in\boldsymbol{Q}(T,\boldsymbol{D})}|Q^-|\Bigr)
\; .
\end{multline*}
Using~(\ref{eq:3}), we get
\begin{multline*}
\sum_{0\leq i< j\leq \mathfrak{q}(T)-1}d(Q^i,Q^j)=
\tfrac{1}{2}\bigl(\mathfrak{q}(T)+1\bigr)
\Bigl(-|T^-|+\sum_{Q\in\boldsymbol{Q}(T,\boldsymbol{D})}|Q^-|\Bigr)
\\=
\tfrac{1}{2}\Bigl(1-\tfrac{2}{t}|T^-|+\tfrac{2}{t}\sum_{Q\in\boldsymbol{Q}(T,\boldsymbol{D})}|Q^-|+1\Bigr)
\Bigl(-|T^-|+\sum_{Q\in\boldsymbol{Q}(T,\boldsymbol{D})}|Q^-|\Bigr)
\; .
\end{multline*}

\begin{remark}
\label{th:2}
Let $\boldsymbol{D}$ be a symmetric cycle in the tope graph of a simple oriented matroid~$\mathcal{M}:=(E_t,\mathcal{T})$.
Let $T\in\mathcal{T}$ be a tope such that $T\not\in\mathrm{V}(\boldsymbol{D})$.

\begin{multline*}
\sum_{0\leq i< j\leq \mathfrak{q}(T)-1}d(Q^i,Q^j)=
\Bigl(1-\tfrac{1}{t}|T^-|+\tfrac{1}{t}\sum_{Q\in\boldsymbol{Q}(T,\boldsymbol{D})}|Q^-|\Bigr)
\\ \times
\Bigl(-|T^-|+\sum_{Q\in\boldsymbol{Q}(T,\boldsymbol{D})}|Q^-|\Bigr)
\; .
\end{multline*}
\end{remark}

For convenience, we now reformulate~\cite[Rem.~13.3]{M-PROM} and Remark~\ref{th:2} via the scalar products of topes:

\begin{remark}Let $\boldsymbol{D}$ be a symmetric cycle in the tope graph of a simple oriented matroid~$\mathcal{M}:=(E_t,\mathcal{T})$.
Let $T\in\mathcal{T}$ be a tope such that $T\not\in\mathrm{V}(\boldsymbol{D})$.
\begin{itemize}
\item[\rm(i)]
\begin{align*}
\sum_{0\leq i< j\leq \mathfrak{q}(T)-1}
\langle Q^i,Q^j\rangle&=\tfrac{1}{2}\bigl(1-\mathfrak{q}(T)\bigr)t\; ,\\
\mathfrak{q}(T)&=1-\tfrac{2}{t}\sum_{0\leq i< j\leq \mathfrak{q}(T)-1}\langle Q^i,Q^j\rangle\; .
\end{align*}

\item[\rm(ii)]
\begin{multline*}
\sum_{0\leq i< j\leq \mathfrak{q}(T)-1}\langle Q^i,Q^j\rangle=
\tbinom{\mathfrak{q}(T)}{2}t -2\Bigl(1-\tfrac{1}{t}|T^-|+\tfrac{1}{t}\sum_{Q\in\boldsymbol{Q}(T,\boldsymbol{D})}|Q^-|\Bigr)
\\ \times
\Bigl(-|T^-|+\sum_{Q\in\boldsymbol{Q}(T,\boldsymbol{D})}|Q^-|\Bigr)
\; .
\end{multline*}
\end{itemize}
\end{remark}

Now, if $T'$ and $T''$ are two topes of~$\mathcal{M}:=(E_t,\mathcal{T})$, then we have
\begin{multline*}
t-2d(T',T'')=\langle T',T''\rangle=\Bigl\langle \sum_{Q'\in\boldsymbol{Q}(T',\boldsymbol{D})}, \sum_{Q''\in\boldsymbol{Q}(T'',\boldsymbol{D})}\Bigr\rangle
=\sum_{\substack{Q'\in\boldsymbol{Q}(T',\boldsymbol{D})\; ,\\
Q''\in\boldsymbol{Q}(T'',\boldsymbol{D})}}\langle Q',Q''\rangle
\\=\sum_{\substack{Q'\in\boldsymbol{Q}(T',\boldsymbol{D})\; ,\\
Q''\in\boldsymbol{Q}(T'',\boldsymbol{D})
}}\bigl(t-2d(Q',Q'')\bigr)=\mathfrak{q}(T')\mathfrak{q}(T'')t-2\sum_{\substack{Q'\in\boldsymbol{Q}(T',\boldsymbol{D})\; ,\\
Q''\in\boldsymbol{Q}(T'',\boldsymbol{D})}} d(Q',Q'')\; .
\end{multline*}

\begin{remark}
If $T',T''\in\mathcal{T}$ are topes of a simple oriented matroid~\mbox{$\mathcal{M}:=(E_t,\mathcal{T})$} with a symmetric cycle~$\boldsymbol{D}$ in its tope graph, then we have
\begin{equation*}
-d(T',T'')+\sum_{\substack{Q'\in\boldsymbol{Q}(T',\boldsymbol{D})\; ,\\
Q''\in\boldsymbol{Q}(T'',\boldsymbol{D})}} d(Q',Q'')=\tfrac{1}{2}\bigl(\mathfrak{q}(T')\mathfrak{q}(T'')-1\bigr)t\; .
\end{equation*}
\end{remark}

\section{Statistics on the negative parts of vertices of the hypercube graph $\boldsymbol{H}(t,2)$ and on the decompositions of vertices}
\label{sec:8}

Throughout this main section we deal exclusively with the distinguished symmetric cycle~$\boldsymbol{R}$, defined by~(\ref{eq:17})(\ref{eq:18}), in the hypercube graph~$\boldsymbol{H}(t,2)$ on the vertex set~$\{1,-1\}^t$.

We present statistics on the decompositions~$\boldsymbol{Q}({}_{-A}\mathrm{T}^{(+)},\boldsymbol{R})$, $\boldsymbol{Q}({}_{-B}\mathrm{T}^{(+)},\boldsymbol{R})$ and $\boldsymbol{Q}(T,\boldsymbol{R})$ of vertices~${}_{-A}\mathrm{T}^{(+)}$, ${}_{-B}\mathrm{T}^{(+)}$ and $T$ such that $T^-=A\cup B$.

The following situations are considered separately:

\noindent$\bullet$ $|A\cap B|=0$ and $|A\,\dot{\cup}\, B|=t$, in Subsection~\ref{sec:1},

\noindent$\bullet$ $|A\cap B|=0$ and $|A\,\dot{\cup}\, B|<t$, in Subsection~\ref{sec:2},

\noindent$\bullet$ $|A\cap B|>0$ and $|A\cup B|=t$, in Subsection~\ref{sec:3}, and

\noindent$\bullet$ $|A\cap B|>0$ and $|A\cup B|<t$, in Subsection~\ref{sec:4}.

Before proceeding to statistics on decompositions, in the next subsection we recall well-known general information on vertex pairs of the discrete hypercube~$\{1,-1\}^t$.

\subsection{Vertex pairs of the discrete hypercube~$\{1,-1\}^t$} $\quad$

We denote by $\#$ the number of tuples in a family.

\begin{remark} Let $j',j''\in[t]$. We have
\begin{itemize}
\item[\rm(i)]
\begin{itemize}
\item[\rm(a)]
\begin{equation*}
\#\bigl\{(X,Y)\in\{1,-1\}^t\times\{1,-1\}^t\colon d(X,Y)=k\bigr\}=2^t\tbinom{t}{k}\; .
\end{equation*}
\item[\rm(b)]
If $t$ is even, then
\begin{equation*}
\#\bigl\{(X,Y)\in\{1,-1\}^t\times\{1,-1\}^t\colon \langle X,Y\rangle=0\bigr\}=2^t\tbinom{t}{t/2}\; .
\end{equation*}
\item[\rm(c)]
The {\em intersection numbers\/}~of the algebraic combinatorial {\em Hamming scheme\/}~$\mathbf{H}(t,2)$~\cite{GM},\cite[\S{}21.3]{MacWS} suggest that for any vertices $X$ and $Y$ of the discrete hypercube $\{1,-1\}^t$ such that \mbox{$\langle X,Y\rangle=0$}, we have
\begin{equation*}
\bigl|\bigl\{Z\in\{1,-1\}^t\colon \langle Z,X\rangle=\langle Z,Y\rangle=0\bigr\}\bigr|
=\begin{cases}
\tbinom{t/2}{t/4}^2\; , & \text{if $4|t$},\\
0\; , & \text{otherwise}.
\end{cases}
\end{equation*}
\end{itemize}

\item[\rm(ii)]
\begin{itemize}
\item[\rm(a)]
\begin{multline*}
\#\bigl\{(X,Y)\in\{1,-1\}^t\times\{1,-1\}^t\colon |X^-|=:j',\ |Y^-|=:j'',\ d(X,Y)=k\bigr\}
\\=
\begin{cases}
\tbinom{t}{k}\tbinom{t-k}{(j'+j''-k)/2}\tbinom{k}{(j'-j''+k)/2}\; , & \text{if $j'+j''+k$ even}\; ,\\
0\; , &  \text{if $j'+j''+k$ odd}\; .
\end{cases}
\end{multline*}
\item[\rm(b)]
If $t$ is even, then
\begin{multline*}
\#\bigl\{(X,Y)\in\{1,-1\}^t\times\{1,-1\}^t\colon |X^-|=:j',\ |Y^-|=:j'',\ \langle X,Y\rangle=0\bigr\}
\\=
\begin{cases}
\tbinom{t}{t/2}\tbinom{t/2}{(2j'+2j''-t)/4}\tbinom{k}{(2j'-2j''+t)/4}, & \text{if $j'+j''+\frac{t}{2}$ even}\; ,\\
0\; , &  \text{if $j'+j''+\frac{t}{2}$ odd}\; .
\end{cases}
\end{multline*}
\end{itemize}
\end{itemize}
\end{remark}

Let $s\in[t]$, and let $\tbinom{E_t}{s}$ denote the family of all subsets $A\subseteq E_t$ with~\mbox{$|A|=s$}.
Recall that if $s\leq 2t$, then the family $\tbinom{E_t}{s}$ can be viewed as the set of elements of the algebraic combinatorial {\em Johnson scheme\/} $\mathbf{J}(t,s)$~\cite{GM}.
Dealing with the $s$th layer $\tbinom{E_t}{s}$ of the Boolean lattice $\mathbb{B}(t)$ of subsets of the set~$E_t$, one often uses a measure of (dis)similarity $\partial(A,B)$ between $s$-sets $A$ and $B$, defined by
\begin{equation*}
\begin{split}
\partial(A,B):\!&=s-|A\cap B|=|A\cup B|-s=\tfrac{1}{2}|A\triangle B|\\
&=\tfrac{1}{2}d({}_{-A}\mathrm{T}^{(+)},{}_{-B}\mathrm{T}^{(+)})=\tfrac{1}{4}\bigl(t-\langle{}_{-A}\mathrm{T}^{(+)},{}_{-B}\mathrm{T}^{(+)}\rangle\bigr)
\; .
\end{split}
\end{equation*}
Note that we have
\begin{equation*}
A\in\tbinom{E_t}{s}\ni B\; ,\ \ \langle{}_{-A}\mathrm{T}^{(+)},{}_{-B}\mathrm{T}^{(+)}\rangle=0\ \ \ \Longleftrightarrow\ \ \
\partial(A,B)=\tfrac{t}{4}\; .
\end{equation*}

\begin{remark} If $s\in[t]$, then we have
\begin{itemize}
\item[\rm(a)]
\begin{multline*}
\#\bigl\{(X,Y)\in\{1,-1\}^t\times\{1,-1\}^t\colon X^-\in\tbinom{E_t}{s}\ni Y^-,\ \partial(X^-,Y^-)=i\bigr\}
\\=
\tbinom{t}{2i}\tbinom{t-2i}{s-i}\tbinom{2i}{i}
\; .
\end{multline*}

\item[\rm(b)]
\begin{multline*}
\#\bigl\{(X,Y)\in\{1,-1\}^t\times\{1,-1\}^t\colon X^-\in\tbinom{E_t}{s}\ni Y^-,\ \langle X,Y\rangle=0\bigr\}
\\=\begin{cases}
\tbinom{t}{t/2}\tbinom{t/2}{s-(t/4)}\tbinom{t/2}{t/4}, & \text{if $4|t$ and $\tfrac{t}{4}\leq s\leq\tfrac{3t}{4}$}\; ,\\
0\; , &  \text{otherwise}\; .
\end{cases}
\end{multline*}

\item[\rm(c)] Suppose that $4|t$, and $\tfrac{t}{4}\leq s\leq\tfrac{3t}{4}$. For any vertices $X$ and $Y$ of the discrete hypercube~$\{1,-1\}^t$ such that~$X^-,Y^-\in\tbinom{E_t}{s}$ and $\langle X,Y\rangle=0$, we have
    \begin{multline}
    \label{eq:21}
    \bigl|\bigl\{Z\in\{1,-1\}^t\colon Z^-\in\tbinom{E_t}{s}\; ,\ \langle Z,X\rangle=\langle Z,Y\rangle=0\bigr\}\bigr|
    \\=\sum_{c=\max\{0,s-(t/2)\}}^{\min\{t/4,s-(t/4)\}}\tbinom{s-(t/4)}{c}\tbinom{(3t/4)-s}{(t/4)-c}\tbinom{t/4}{s-(t/4)-c}^{\!2}\; .
    \end{multline}
        The quantity~{\rm(\ref{eq:21})} is suggested, in the case $2s\leq t$, by the {\em intersection numbers} of the {\em Johnson scheme} $\mathbf{J}(t,s)$.
\end{itemize}
\end{remark}

\subsection{Case: $|A\cap B|=0$ and $|A\,\dot{\cup}\,B|=t$}  $\quad$
\label{sec:1}

Let us consider the family
\begin{multline}
\label{eq:5}
\bigl\{(A,B)\in\mathbf{2}^{[t]}\times\mathbf{2}^{[t]}
\colon\ \ 0<|A|=:j'\; ,\ \  0<|B|=:j''\; ,\\
|A\cap B|=0\; ,\ \ |A\,\dot{\cup}\,B|=t\; ,\ \
\mathfrak{q}(A)=\ell'\; ,\ \
\mathfrak{q}(B)=\ell''
\bigr\}
\end{multline}
of ordered two-member {\em partitions\/} of the set $E_t$; thus, $j'+j''=t$. The integers $\ell',\ell''\in[t]$ are {\em odd\/}. For any pair~$(A,B)$ in this family, we have
\begin{equation*}
d({}_{-A}\mathrm{T}^{(+)},{}_{-B}\mathrm{T}^{(+)})=t\ \ \ \text{and}\ \ \
\langle {}_{-A}\mathrm{T}^{(+)},{}_{-B}\mathrm{T}^{(+)}\rangle=-t\; .
\end{equation*}

Note that if $\{1,t\}\cap A=\{1\}$ and $\{1,t\}\cap B=\{t\}$, with $\mathfrak{q}(A)=\ell'$ and~$\mathfrak{q}(B)=\ell''$, then we have
\begin{equation*}
\ \ \ \varrho(A)=\tfrac{\ell'+1}{2}\; ,\ \ \ \varrho(B)=\tfrac{\ell''+1}{2}\; ,
\end{equation*}
by~\cite[Prop.~2.4(i)(iv)]{M-SC-II}.

If $\{1,t\}\cap A=\{1,t\}$ and $|\{1,t\}\cap B|=0$, then we have
\begin{equation*}
\ \ \ \varrho(A)=\tfrac{\ell'+1}{2}\; ,\ \ \ \varrho(B)=\tfrac{\ell''-1}{2}\; ,
\end{equation*}
by~\cite[Prop.~2.4(ii)(iii)]{M-SC-II}.

\begin{remark} Ordered pairs of sets~$(A,B)$ in the family~{\rm(\ref{eq:5})} can be counted with the help of products of the form
\begin{equation*}
\mathtt{c}(\varrho(A);j')\cdot\mathtt{c}(\varrho(B);j'')\; :
\end{equation*}
\begin{itemize}
\item[\rm(i)]
In the family~{\rm(\ref{eq:5})} there are
\begin{equation}
\label{eq:14}
\mathtt{c}(\tfrac{\ell'+1}{2};j')\cdot\mathtt{c}(\tfrac{\ell''+1}{2};j'')
\end{equation}
pairs~$(A,B)$ of sets $A$ and $B$ such that
\begin{equation*}
\{1,t\}\cap A=\{1\}\ \ \ \text{and}\ \ \ \{1,t\}\cap B=\{t\}\; ,
\end{equation*}
and if the subfamily of these pairs is nonempty, then $\ell'=\ell''$.

\item[\rm(ii)]
In the family~{\rm(\ref{eq:5})} there are
\begin{equation*}
\mathtt{c}(\tfrac{\ell'+1}{2};j')\cdot\mathtt{c}(\tfrac{\ell''-1}{2};j'')
\end{equation*}
pairs~$(A,B)$ such that
\begin{equation*}
\{1,t\}\cap A=\{1,t\}\ \ \ \text{and}\ \ \ |\{1,t\}\cap B|=0\; ,
\end{equation*}
and if the subfamily of these pairs is nonempty, then $\ell'=\ell''$.
\end{itemize}
\end{remark}

\subsection{Case: $|A\cap B|=0$ and $|A\,\dot{\cup}\,B|<t$}  $\quad$
\label{sec:2}

Let us consider the family
\begin{multline}
\label{eq:6}
\bigl\{(A,B)\in\mathbf{2}^{[t]}\times\mathbf{2}^{[t]}
\colon\ \ 0<|A|=:j'\; ,\ \  0<|B|=:j''\; ,\\
|A\cap B|=0\; ,\ \ |A\,\dot{\cup}\,B|<t\; ,\ \
\mathfrak{q}(A)=\ell'\; ,\ \
\mathfrak{q}(B)=\ell''\; ,\ \
\mathfrak{q}(A\,\dot{\cup}\,B)=\ell
\bigr\}
\end{multline}
of ordered pairs of {\em disjoint\/} subsets of the set~$E_t$, with the {\em odd\/} integers~\mbox{$\ell',\ell'',\ell\in[t]$}.
Note that for pairs~$(A,B)$ in this family, we have
\begin{equation*}
d({}_{-A}\mathrm{T}^{(+)},{}_{-B}\mathrm{T}^{(+)})=j'+j''\; ,
\end{equation*}
and
\begin{equation*}
\langle {}_{-A}\mathrm{T}^{(+)},{}_{-B}\mathrm{T}^{(+)}\rangle=0\ \ \ \Longleftrightarrow\ \ \ 2(j'+j'')=t\; .
\end{equation*}

\begin{proposition}{\rm\cite[Th.~4.1, rephrased and abridged]{M-SC-III}}{\bf:}
Ordered pairs of sets $(A,B)$ in the family~{\rm(\ref{eq:6})} can be counted with the help of products of the form
\begin{multline*}
\mathfrak{T}\left(\mathfrak{s}',\; \mathfrak{s}'';\; \varrho(E_t-(A\,\dot{\cup}\,B)),\; \varrho(A),\; \varrho(B)\right)
\\ \times\mathtt{c}\bigl(\varrho(E_t-(A\,\dot{\cup}\,B));t-(j'+j'')\bigr)\cdot
\mathtt{c}(\varrho(A);j')
\cdot\mathtt{c}(\varrho(B);j''):
\end{multline*}
\begin{itemize}
\item[\rm(i)] In the family~{\rm(\ref{eq:6})} there are
\begin{equation*}
\mathfrak{T}\left(\theta,\theta;\tfrac{\ell+1}{2},\tfrac{\ell'-1}{2}, \tfrac{\ell''-1}{2}\right)
\cdot\mathtt{c}\bigl(\tfrac{\ell+1}{2};t-(j'+j'')\bigr)\cdot
\mathtt{c}(\tfrac{\ell'-1}{2};j')
\cdot\mathtt{c}(\tfrac{\ell''-1}{2};j'')
\end{equation*}
pairs~$(A,B)$ of sets $A$ and $B$ such that
\begin{equation*}
|\{1,t\}\cap A|=|\{1,t\}\cap B|=0\; .
\end{equation*}

In the family~{\rm(\ref{eq:6})} there are
\item[\rm(ii)]
\begin{equation*}
\mathfrak{T}\left(\alpha,\alpha;\tfrac{\ell-1}{2}, \tfrac{\ell'+1}{2}, \tfrac{\ell''-1}{2}\right)\cdot
\mathtt{c}\bigl(\tfrac{\ell-1}{2};t-(j'+j'')\bigr)\cdot
\mathtt{c}(\tfrac{\ell'+1}{2};j')
\cdot\mathtt{c}(\tfrac{\ell''-1}{2};j'')
\end{equation*}
pairs~$(A,B)$  such that
\begin{equation}
\{1,t\}\cap A=\{1,t\}\ \ \ \text{and}\ \ \ |\{1,t\}\cap B|=0\; ;
\end{equation}

\item[\rm(iii)]
\begin{equation*}
\mathfrak{T}\left(\theta,\beta;\tfrac{\ell+1}{2},\tfrac{\ell'-1}{2},\tfrac{\ell''+1}{2}\right)\cdot
\mathtt{c}\bigl(\tfrac{\ell+1}{2};t-(j'+j'')\bigr)\cdot
\mathtt{c}(\tfrac{\ell'-1}{2};j')
\cdot\mathtt{c}(\tfrac{\ell''+1}{2};j'')
\end{equation*}
pairs~$(A,B)$  such that
\begin{equation*}
|\{1,t\}\cap A|=0\ \ \ \text{and} \ \ \ \{1,t\}\cap B=\{t\}\; ;
\end{equation*}

\item[\rm(iv)]
\begin{equation*}
\mathfrak{T}\left(\alpha,\theta;\tfrac{\ell+1}{2},\tfrac{\ell'+1}{2}, \tfrac{\ell''-1}{2}\right)
\cdot\mathtt{c}\bigl(\tfrac{\ell+1}{2};t-(j'+j'')\bigr)\cdot
\mathtt{c}(\tfrac{\ell'+1}{2};j')
\cdot\mathtt{c}(\tfrac{\ell''-1}{2};j'')
\end{equation*}
pairs~$(A,B)$  such that
\begin{equation*}
\{1,t\}\cap A=\{1\}\ \ \ \text{and}\ \ \ |\{1,t\}\cap B|=0\; ;
\end{equation*}

\item[\rm(v)]
\begin{equation*}
\mathfrak{T}\left(\alpha,\beta; \tfrac{\ell-1}{2}, \tfrac{\ell'+1}{2}, \tfrac{\ell''+1}{2}\right)\cdot
\mathtt{c}\bigl(\tfrac{\ell-1}{2};t-(j'+j'')\bigr)\cdot
\mathtt{c}(\tfrac{\ell'+1}{2};j')
\cdot\mathtt{c}(\tfrac{\ell''+1}{2};j'')
\end{equation*}
pairs~$(A,B)$  such that
\begin{equation*}
\{1,t\}\cap A=\{1\}\ \ \ \text{and}\ \ \ \{1,t\}\cap B=\{t\}\; .
\end{equation*}
\end{itemize}
\end{proposition}

\subsection{Case: $|A\cap B|>0$ and $|A\cup B|=t$}  $\quad$
\label{sec:3}

\noindent$\bullet$ Here we consider the family
\begin{multline}
\label{eq:7}
\bigl\{(A,B)\in\mathbf{2}^{[t]}\times\mathbf{2}^{[t]}
\colon\ \ |A|=:j'\; ,\ \  |B|=:j''\; ,\\
0<|A\cap B|
\; ,\ \ \max\{j',j''\}<|A\cup B|=t\; ,\\
\mathfrak{q}(A\cap B)=\ell^{\cap}\; ,\ \
\mathfrak{q}(A-B)=\ell'\; ,\ \
\mathfrak{q}(B-A)=\ell''
\bigr\}
\end{multline}
of ordered two-member {\em intersecting Sperner families\/} that {\em cover\/} the set $E_t$; the integers~$\ell^{\cap},\ell',\ell''\in[t]$ are {\em odd}. Note that for pairs~$(A,B)$ in this family, we have
\begin{equation}
\label{eq:24}
d({}_{-A}\mathrm{T}^{(+)},{}_{-B}\mathrm{T}^{(+)})=2t-j'-j''\; ,
\end{equation}
and
\begin{equation}
\label{eq:25}
\langle {}_{-A}\mathrm{T}^{(+)},{}_{-B}\mathrm{T}^{(+)}\rangle=0\ \ \ \Longleftrightarrow\ \ \ 2(j'+j'')=3t\; .
\end{equation}

\begin{theorem}
Ordered pairs of sets $(A,B)$ in the family~{\rm(\ref{eq:7})} can be counted with the help of products of the form
\begin{multline}
\label{eq:20}
\mathfrak{T}\left(\mathfrak{s}',\; \mathfrak{s}'';\; \varrho(A \cap B),\; \varrho(A-B),\; \varrho(B-A)\right)
\\ \times \mathtt{c}\bigl(\varrho(A \cap B);(j'+j'')-t\bigr)
\\ \times \mathtt{c}\bigl(\varrho(A-B);t-j''\bigr)
\cdot \mathtt{c}\bigl(\varrho(B-A);t-j'\bigr)\; :
\end{multline}
\begin{itemize}
\item[\rm(i)] In the family~{\rm(\ref{eq:7})} there are
\begin{equation*}
\mathfrak{T}\left(\theta,\beta;\tfrac{\ell^{\cap}+1}{2},\tfrac{\ell'-1}{2},\tfrac{\ell''+1}{2}\right)
\cdot \mathtt{c}\bigl(\tfrac{\ell^{\cap}+1}{2};(j'+j'')-t\bigr)\cdot \mathtt{c}\bigl(\tfrac{\ell'-1}{2};t-j''\bigr)
\cdot \mathtt{c}\bigl(\tfrac{\ell''+1}{2};t-j'\bigr)
\end{equation*}
pairs $(A,B)$ of sets $A$ and $B$ such that
\begin{equation*}
\{1,t\}\cap A=\{1\}\ \ \ \text{and}\ \ \ \{1,t\}\cap B=\{1,t\}\; .
\end{equation*}

In the family~{\rm(\ref{eq:7})} there are
\item[\rm(ii)]
\begin{equation*}
\mathfrak{T}\left(\alpha,\beta;\tfrac{\ell^{\cap}-1}{2},\tfrac{\ell'+1}{2},\tfrac{\ell''+1}{2}\right)
\cdot \mathtt{c}\bigl(\tfrac{\ell^{\cap}-1}{2};(j'+j'')-t\bigr)\cdot \mathtt{c}\bigl(\tfrac{\ell'+1}{2};t-j''\bigr)
\cdot \mathtt{c}\bigl(\tfrac{\ell''+1}{2};t-j'\bigr)
\end{equation*}
pairs $(A,B)$ such that
\begin{equation*}
\{1,t\}\cap A=\{1\}\ \ \ \text{and}\ \ \ \{1,t\}\cap B=\{t\}\; ;
\end{equation*}

\item[\rm(iii)]
\begin{equation*}
\mathfrak{T}\left(\theta,\theta;\tfrac{\ell^{\cap}+1}{2},\tfrac{\ell'-1}{2},\tfrac{\ell''-1}{2}\right)
\cdot \mathtt{c}\bigl(\tfrac{\ell^{\cap}+1}{2};(j'+j'')-t\bigr)\cdot \mathtt{c}\bigl(\tfrac{\ell'-1}{2};t-j''\bigr)
\cdot \mathtt{c}\bigl(\tfrac{\ell''-1}{2};t-j'\bigr)
\end{equation*}
pairs $(A,B)$ such that
\begin{equation*}
\{1,t\}\cap A=\{1,t\}\cap B=\{1,t\}\; ;
\end{equation*}

\item[\rm(iv)]
\begin{equation*}
\mathfrak{T}\left(\alpha,\alpha;\tfrac{\ell^{\cap}-1}{2},\tfrac{\ell'+1}{2},\tfrac{\ell''-1}{2}\right)
\cdot \mathtt{c}\bigl(\tfrac{\ell^{\cap}-1}{2};(j'+j'')-t\bigr)\cdot \mathtt{c}\bigl(\tfrac{\ell'+1}{2};t-j''\bigr)
\cdot \mathtt{c}\bigl(\tfrac{\ell''-1}{2};t-j'\bigr)
\end{equation*}
pairs $(A,B)$ such that
\begin{equation*}
\{1,t\}\cap A=\{1,t\}\ \ \ \text{and}\ \ \ |\{1,t\}\cap B|=0\; ;
\end{equation*}

\item[\rm(v)]
\begin{equation*}
\mathfrak{T}\left(\alpha,\theta;\tfrac{\ell^{\cap}+1}{2},\tfrac{\ell'+1}{2},\tfrac{\ell''-1}{2}\right)
\cdot \mathtt{c}\bigl(\tfrac{\ell^{\cap}+1}{2};(j'+j'')-t\bigr)\cdot \mathtt{c}\bigl(\tfrac{\ell'+1}{2};t-j''\bigr)
\cdot \mathtt{c}\bigl(\tfrac{\ell''-1}{2};t-j'\bigr)
\end{equation*}
pairs $(A,B)$ such that
\begin{equation*}
\{1,t\}\cap A=\{1,t\}\ \ \ \text{and}\ \ \ \{1,t\}\cap B=\{t\}\; .
\end{equation*}
\end{itemize}
\end{theorem}

\begin{proof}
For each pair $(A,B)$ in the family~{\rm(\ref{eq:7})}, pick an {\em ascending system\/} of {\em distinct
representatives}~$(e_1<e_2<\ldots<e_{\varrho(A\cap B)+\varrho(A-B)+\varrho(B-A)})\subseteq E_t$ of the
intervals composing the
sets~$(A\cap B)$, $(A-B)$ and~$(B-A)$. By making the substitutions
\begin{equation*}
e_i\mapsto
\begin{cases}
\theta\; , & \text{if $e_i\in A\cap B$},\\
\alpha\; , & \text{if $e_i\in A-B$},\\
\beta\; , & \text{if $e_i\in B-A$},
\end{cases}
\ \ \ \ \ 1\leq i\leq\varrho(A\cap B)+\varrho(A-B)+\varrho(B-A)\; ,
\end{equation*}
we obtain some ternary Smirnov words over the alphabet~$(\theta,\alpha,\beta)$. Given any such Smirnov word, we find the number of pairs~$(A,B)$ in the corresponding subfamily of the family~{\rm(\ref{eq:7})} by means of a product of the form~(\ref{eq:20}).

\begin{itemize}
\item[\rm(i)]
Since $\{1,t\}\cap A=\{1\}$, and $\{1,t\}\cap B=\{1,t\}$, we have
\begin{equation*}
\ \ \ \varrho(A-B)=\tfrac{\ell'-1}{2}\; ,\ \ \ \varrho(B-A)=\tfrac{\ell''+1}{2}\; ,
\end{equation*}
by~\cite[Prop.~2.4(iii)(iv)]{M-SC-II}, and
\begin{equation*}
\ \ \ \varrho(A\cap B)=\tfrac{\ell^{\cap}+1}{2}\; ,
\end{equation*}
by~\cite[Prop.~2.4(i)]{M-SC-II}.

\item[\rm(ii)]
Since $\{1,t\}\cap A=\{1\}$, and $\{1,t\}\cap B=\{t\}$, we have
\begin{equation*}
\ \ \ \varrho(A-B)=\tfrac{\ell'+1}{2}\; ,\ \ \ \varrho(B-A)=\tfrac{\ell''+1}{2}\; ,
\end{equation*}
by~\cite[Prop.~2.4(i)(iv)]{M-SC-II}, and
\begin{equation*}
\ \ \ \varrho(A\cap B)=\tfrac{\ell^{\cap}-1}{2}\; ,
\end{equation*}
by~\cite[Prop.~2.4(iii)]{M-SC-II}.

\item[\rm(iii)]
Since $\{1,t\}\cap A=\{1,t\}\cap B=\{1,t\}$, we have
\begin{equation*}
\ \ \ \varrho(A-B)=\tfrac{\ell'-1}{2}\; ,\ \ \ \varrho(B-A)=\tfrac{\ell''-1}{2}\; ,
\end{equation*}
by~\cite[Prop.~2.4(iii)]{M-SC-II}, and
\begin{equation*}
\ \ \ \varrho(A\cap B)=\tfrac{\ell^{\cap}+1}{2}\; ,
\end{equation*}
by~\cite[Prop.~2.4(ii)]{M-SC-II}.

\item[\rm(iv)]
Since $\{1,t\}\cap A=\{1,t\}$, and $|\{1,t\}\cap B|=0$, we have
\begin{equation*}
\ \ \ \varrho(A-B)=\tfrac{\ell'+1}{2}\; ,\ \ \ \varrho(B-A)=\tfrac{\ell''-1}{2}\; ,
\end{equation*}
by~\cite[Prop.~2.4(ii)(iii)]{M-SC-II}, and
\begin{equation*}
\ \ \ \varrho(A\cap B)=\tfrac{\ell^{\cap}-1}{2}\; ,
\end{equation*}
by~\cite[Prop.~2.4(iii)]{M-SC-II}.

\item[\rm(v)]
Since $\{1,t\}\cap A=\{1,t\}$, and $\{1,t\}\cap B=\{t\}$, we have
\begin{equation*}
\ \ \ \varrho(A-B)=\tfrac{\ell'+1}{2}\; ,\ \ \ \varrho(B-A)=\tfrac{\ell''-1}{2}\; ,
\end{equation*}
by~\cite[Prop.~2.4(i)(iii)]{M-SC-II}, and
\begin{equation*}
\ \ \ \varrho(A\cap B)=\tfrac{\ell^{\cap}+1}{2}\; ,
\end{equation*}
by~\cite[Prop.~2.4(iv)]{M-SC-II}.
\end{itemize}
\end{proof}

\noindent$\bullet$ Let us consider the
family
\begin{multline}
\label{eq:11}
\bigl\{(A,B)\in\mathbf{2}^{[t]}\times\mathbf{2}^{[t]}
\colon\ \ |A|=:j'\; ,\ \  |B|=:j''\; ,\\
0<|A\cap B|
\; ,\ \ \max\{j',j''\}<|A\cup B|=t\; ,\\
\mathfrak{q}(A\cap B)=\ell^{\cap}\; ,\ \
\mathfrak{q}(A\triangle B)=\ell^{\triangle}
\bigr\}
\end{multline}
of ordered two-member {\em intersecting Sperner families\/} that {\em cover\/} the set $E_t$, with the properties~(\ref{eq:24})(\ref{eq:25}).

\begin{theorem}
\label{th:3}
Ordered pairs of sets~$(A,B)$ in the family~{\rm(\ref{eq:11})} can be counted with the help of products that include subproducts of the form
\begin{equation*}
\mathtt{c}\bigl(\varrho(A\cap B);(j'+j'')-t\bigr)\cdot\mathtt{c}\bigl(\varrho(A\triangle B);2t-(j'+j'')\bigr)\; :
\end{equation*}
\begin{itemize}
\item[\rm(i)]
In the family~{\rm(\ref{eq:11})} there are
\begin{equation*}
\mathtt{c}\bigl(\tfrac{\ell^{\cap}+1}{2};(j'+j'')-t\bigr)\cdot\mathtt{c}\bigl(\tfrac{\ell^{\triangle}+1}{2};2t-(j'+j'')\bigr)
\cdot\tbinom{2t-(j'+j'')-1}{t-j''}
\end{equation*}
pairs~$(A,B)$ of sets $A$ and $B$ such that
\begin{equation*}
\{1,t\}\cap A=\{1\}\ \ \ \text{and}\ \ \ \{1,t\}\cap B=\{1,t\}\; ,
\end{equation*}
and if the subfamily of these pairs is nonempty, then $\ell^{\cap}=\ell^{\triangle}$.
\end{itemize}

In the family~{\rm(\ref{eq:11})} there are
\begin{itemize}
\item[\rm(ii)]
\begin{equation*}
\mathtt{c}\bigl(\tfrac{\ell^{\cap}-1}{2};(j'+j'')-t\bigr)\cdot\mathtt{c}\bigl(\tfrac{\ell^{\triangle}+1}{2};2t-(j'+j'')\bigr)
\cdot\tbinom{2t-(j'+j'')-2}{t-j''-1}
\end{equation*}
pairs~$(A,B)$ such that
\begin{equation*}
\{1,t\}\cap A=\{1\}\ \ \ \text{and}\ \ \ \{1,t\}\cap B=\{t\}\; ,
\end{equation*}
and if the subfamily of these pairs is nonempty, then $\ell^{\cap}=\ell^{\triangle}$.

\item[\rm(iii)]
\begin{equation*}
\mathtt{c}\bigl(\tfrac{\ell^{\cap}+1}{2};(j'+j'')-t\bigr)\cdot\mathtt{c}\bigl(\tfrac{\ell^{\triangle}-1}{2};2t-(j'+j'')\bigr)
\cdot\tbinom{2t-(j'+j'')}{t-j''}
\end{equation*}
pairs~$(A,B)$ such that
\begin{equation*}
\{1,t\}\cap A=\{1,t\}\cap B=\{1,t\}\;
\end{equation*}
and if the subfamily of these pairs is nonempty, then $\ell^{\cap}=\ell^{\triangle}$.

\item[\rm(iv)]
\begin{equation*}
\mathtt{c}\bigl(\tfrac{\ell^{\cap}-1}{2};(j'+j'')-t\bigr)\cdot\mathtt{c}\bigl(\tfrac{\ell^{\triangle}+1}{2};2t-(j'+j'')\bigr)
\cdot\tbinom{2t-(j'+j'')-2}{t-j'}
\end{equation*}
pairs~$(A,B)$ such that
\begin{equation*}
\{1,t\}\cap A=\{1,t\}\ \ \ \text{and}\ \ \ |\{1,t\}\cap B|=0\; ,
\end{equation*}
and if the subfamily of these pairs is nonempty, then $\ell^{\cap}=\ell^{\triangle}$.

\item[\rm(v)]
\begin{equation*}
\mathtt{c}\bigl(\tfrac{\ell^{\cap}+1}{2};(j'+j'')-t\bigr)\cdot\mathtt{c}\bigl(\tfrac{\ell^{\triangle}+1}{2};2t-(j'+j'')\bigr)
\cdot\tbinom{2t-(j'+j'')-1}{t-j'}
\end{equation*}
pairs~$(A,B)$ such that
\begin{equation*}
\{1,t\}\cap A=\{1,t\}\ \ \ \text{and}\ \ \ \{1,t\}\cap B=\{t\}\; ,
\end{equation*}
and if the subfamily of these pairs is nonempty, then $\ell^{\cap}=\ell^{\triangle}$.
\end{itemize}
\end{theorem}

\begin{proof}
\begin{itemize}
\item[\rm(i)]
Since $\{1,t\}\cap A=\{1\}$, and $\{1,t\}\cap B=\{1,t\}$, we have
\begin{equation*}
\ \ \ \varrho(A\cap B)=\tfrac{\ell^{\cap}+1}{2}\; ,\ \ \ \varrho(A\triangle B)=\tfrac{\ell^{\triangle}+1}{2}\; ,
\end{equation*}
by~\cite[Prop.~2.4(i)(iv)]{M-SC-II}.

\item[\rm(ii)]
Since $\{1,t\}\cap A=\{1\}$, and $\{1,t\}\cap B=\{t\}$, we have
\begin{equation*}
\ \ \ \varrho(A\cap B)=\tfrac{\ell^{\cap}-1}{2}\; ,\ \ \ \varrho(A\triangle B)=\tfrac{\ell^{\triangle}+1}{2}\; ,
\end{equation*}
by~\cite[Prop.~2.4(iii)(ii)]{M-SC-II}.

\item[\rm(iii)]
Since $\{1,t\}\cap A=\{1,t\}\cap B=\{1,t\}$,
we have
\begin{equation*}
\ \ \ \varrho(A\cap B)=\tfrac{\ell^{\cap}+1}{2}\; ,\ \ \ \varrho(A\triangle B)=\tfrac{\ell^{\triangle}-1}{2}\; ,
\end{equation*}
by~\cite[Prop.~2.4(ii)(iii)]{M-SC-II}.

\item[\rm(iv)]
Since $\{1,t\}\cap A=\{1,t\}$, and $|\{1,t\}\cap B|=0$, we have
\begin{equation*}
\ \ \ \varrho(A\cap B)=\tfrac{\ell^{\cap}-1}{2}\; ,\ \ \ \varrho(A\triangle B)=\tfrac{\ell^{\triangle}+1}{2}\; ,
\end{equation*}
by~\cite[Prop.~2.4(iii)(ii)]{M-SC-II}.

\item[\rm(v)]
Since $\{1,t\}\cap A=\{1,t\}$, and $\{1,t\}\cap B=\{t\}$,
we have
\begin{equation*}
\ \ \ \varrho(A\cap B)=\tfrac{\ell^{\cap}+1}{2}\; ,\ \ \ \varrho(A\triangle B)=\tfrac{\ell^{\triangle}+1}{2}\; ,
\end{equation*}
by~\cite[Prop.~2.4(iv)(i)]{M-SC-II}.
\end{itemize}
\end{proof}

\subsection{Case: $|A\cap B|>0$ and $|A\cup B|<t$}  $\quad$
\label{sec:4}

\noindent$\bullet$ Here we consider the family
\begin{multline}
\label{eq:8}
\bigl\{(A,B)\in\mathbf{2}^{[t]}\times\mathbf{2}^{[t]}
\colon\ \ |A|=:j'<t\; ,\ \  |B|=:j''<t\; ,\\
0<|A\cap B|=:j\; ,\ \ \max\{j',j''\}<|A\cup B|<t\; ,\\
\mathfrak{q}(A-B)=\ell'\; ,\ \
\mathfrak{q}(B-A)=\ell''\; ,\ \
\mathfrak{q}(A\cap B)=\ell^{\cap}\; ,\ \
\mathfrak{q}(A\cup B)=\ell
\bigr\}
\end{multline}
of ordered two-member {\em intersecting Sperner families\/} that {\em do not cover\/} the set $E_t$. For pairs in this family, we have
\begin{equation}
\label{eq:26}
d({}_{-A}\mathrm{T}^{(+)},{}_{-B}\mathrm{T}^{(+)})=j'+j''-2j\; ,
\end{equation}
and
\begin{equation}
\label{eq:27}
\langle {}_{-A}\mathrm{T}^{(+)},{}_{-B}\mathrm{T}^{(+)}\rangle=0\ \ \ \Longleftrightarrow\ \ \ 2(j'+j''-2j)=t\; .
\end{equation}

\begin{theorem}
Ordered pairs of sets $(A,B)$ in the family~{\rm(\ref{eq:8})} can be counted with the help of products of the form
\begin{multline}
\label{eq:22}
\mathfrak{F}\left(\mathfrak{s}',\; \mathfrak{s}'';\; \varrho(E_t-(A\cup B)),\; \varrho(A-B),\; \varrho(B-A),\; \varrho(A \cap B)\right)
\\ \times \mathtt{c}\bigl(\varrho(E_t-(A\cup B));t-(j'+j''-j)\bigr)
\\ \times \mathtt{c}\bigl(\varrho(A-B);j'-j\bigr)
\cdot \mathtt{c}\bigl(\varrho(B-A);j''-j\bigr)\cdot \mathtt{c}\bigl(\varrho(A \cap B);j\bigr)\; :
\end{multline}

\begin{itemize}
\item[\rm(i)]
In the family~{\rm(\ref{eq:8})} there are
\begin{equation*}
\mathfrak{F}\left(\gamma,\theta;\tfrac{\ell+1}{2},\tfrac{\ell'-1}{2},\tfrac{\ell''-1}{2},\tfrac{\ell^{\cap}+1}{2}\right)
\cdot \mathtt{c}\bigl(\tfrac{\ell+1}{2};t-(j'+j''-j)\bigr)\cdot \mathtt{c}\bigl(\tfrac{\ell'-1}{2};j'-j\bigr)
\cdot \mathtt{c}\bigl(\tfrac{\ell''-1}{2};j''-j\bigr)\cdot \mathtt{c}\bigl(\tfrac{\ell^{\cap}+1}{2};j\bigr)
\end{equation*}
pairs $(A,B)$ of sets $A$ and $B$ such that
\begin{equation*}
\{1,t\}\cap A =\{1,t\}\cap B=\{1\}\; .
\end{equation*}
\end{itemize}

In the family~{\rm(\ref{eq:8})} there are
\begin{itemize}
\item[\rm(ii)]
\begin{equation*}
\mathfrak{F}\left(\gamma,\beta;\tfrac{\ell-1}{2},\tfrac{\ell'-1}{2},\tfrac{\ell''+1}{2},\tfrac{\ell^{\cap}+1}{2}\right)
\cdot \mathtt{c}\bigl(\tfrac{\ell-1}{2}; t-(j'+j''-j)\bigr)\cdot \mathtt{c}\bigl(\tfrac{\ell'-1}{2};j'-j\bigr)
\cdot \mathtt{c}\bigl(\tfrac{\ell''+1}{2};j''-j\bigr)\cdot \mathtt{c}\bigl(\tfrac{\ell^{\cap}+1}{2};j\bigr)
\end{equation*}
pairs $(A,B)$ such that
\begin{equation*}
\{1,t\}\cap A=\{1\}\ \ \ \text{and}\ \ \ \{1,t\}\cap B=\{1,t\}\; ;
\end{equation*}

\item[\rm(iii)]
\begin{equation*}
\mathfrak{F}\left(\alpha,\theta;\tfrac{\ell+1}{2},\tfrac{\ell'+1}{2},\tfrac{\ell''-1}{2},\tfrac{\ell^{\cap}-1}{2}\right)
\cdot \mathtt{c}\bigl(\tfrac{\ell+1}{2}; t-(j'+j''-j)\bigr)\cdot \mathtt{c}\bigl(\tfrac{\ell'+1}{2};j'-j\bigr)
\cdot \mathtt{c}\bigl(\tfrac{\ell''-1}{2};j''-j\bigr)\cdot \mathtt{c}\bigl(\tfrac{\ell^{\cap}-1}{2};j\bigr)
\end{equation*}
pairs $(A,B)$ such that
\begin{equation*}
\{1,t\}\cap A=\{1\}\ \ \ \text{and}\ \ \ |\{1,t\}\cap B|=0\; ;
\end{equation*}

\item[\rm(iv)]
\begin{equation*}
\mathfrak{F}\left(\alpha,\beta;\tfrac{\ell-1}{2},\tfrac{\ell'+1}{2},\tfrac{\ell''+1}{2},\tfrac{\ell^{\cap}-1}{2}\right)
\cdot \mathtt{c}\bigl(\tfrac{\ell-1}{2}; t-(j'+j''-j)\bigr)\cdot \mathtt{c}\bigl(\tfrac{\ell'+1}{2};j'-j\bigr)
\cdot \mathtt{c}\bigl(\tfrac{\ell''+1}{2};j''-j\bigr)\cdot \mathtt{c}\bigl(\tfrac{\ell^{\cap}-1}{2};j\bigr)
\end{equation*}
pairs $(A,B)$ such that
\begin{equation*}
\{1,t\}\cap A=\{1\}\ \ \ \text{and}\ \ \ \{1,t\}\cap B=\{t\}\; ;
\end{equation*}

\item[\rm(v)]
\begin{equation*}
\mathfrak{F}\left(\gamma,\gamma;\tfrac{\ell-1}{2},\tfrac{\ell'-1}{2},\tfrac{\ell''-1}{2},\tfrac{\ell^{\cap}+1}{2}\right)
\cdot \mathtt{c}\bigl(\tfrac{\ell-1}{2};t-(j'+j''-j)\bigr)\cdot \mathtt{c}\bigl(\tfrac{\ell'-1}{2};j'-j\bigr)
\cdot \mathtt{c}\bigl(\tfrac{\ell''-1}{2};j''-j\bigr)\cdot \mathtt{c}\bigl(\tfrac{\ell^{\cap}+1}{2};j\bigr)
\end{equation*}
pairs $(A,B)$ such that
\begin{equation*}
\{1,t\}\cap A=\{1,t\}\cap B=\{1,t\}\; ;
\end{equation*}

\item[\rm(vi)]
\begin{equation*}
\mathfrak{F}\left(\alpha,\alpha;\tfrac{\ell-1}{2},\tfrac{\ell'+1}{2},\tfrac{\ell''-1}{2},\tfrac{\ell^{\cap}-1}{2}\right)
\cdot \mathtt{c}\bigl(\tfrac{\ell-1}{2}; t-(j'+j''-j)\bigr)\cdot \mathtt{c}\bigl(\tfrac{\ell'+1}{2};j'-j\bigr)
\cdot \mathtt{c}\bigl(\tfrac{\ell''-1}{2};j''-j\bigr)\cdot \mathtt{c}\bigl(\tfrac{\ell^{\cap}-1}{2};j\bigr)
\end{equation*}
pairs $(A,B)$ such that
\begin{equation*}
\{1,t\}\cap A=\{1,t\}\ \ \ \text{and}\ \ \ |\{1,t\}\cap B|=0\; ;
\end{equation*}

\item[\rm(vii)]
\begin{equation*}
\mathfrak{F}\left(\alpha,\gamma;\tfrac{\ell-1}{2},\tfrac{\ell'+1}{2},\tfrac{\ell''-1}{2},\tfrac{\ell^{\cap}+1}{2}\right)
\cdot \mathtt{c}\bigl(\tfrac{\ell-1}{2}; t-(j'+j''-j)\bigr)\cdot \mathtt{c}\bigl(\tfrac{\ell'+1}{2};j'-j\bigr)
\cdot \mathtt{c}\bigl(\tfrac{\ell''-1}{2};j''-j\bigr)\cdot \mathtt{c}\bigl(\tfrac{\ell^{\cap}+1}{2};j\bigr)
\end{equation*}
pairs $(A,B)$ such that
\begin{equation*}
\{1,t\}\cap A=\{1,t\}\ \ \ \text{and}\ \ \ \{1,t\}\cap B=\{t\}\; ;
\end{equation*}

\item[\rm(viii)]
\begin{equation*}
\mathfrak{F}\left(\theta,\theta;\tfrac{\ell+1}{2},\tfrac{\ell'-1}{2},\tfrac{\ell''-1}{2},\tfrac{\ell^{\cap}-1}{2}\right)
\cdot \mathtt{c}\bigl(\tfrac{\ell+1}{2};t-(j'+j''-j)\bigr)\cdot \mathtt{c}\bigl(\tfrac{\ell'-1}{2};j'-j\bigr)
\cdot \mathtt{c}\bigl(\tfrac{\ell''-1}{2};j''-j\bigr)\cdot \mathtt{c}\bigl(\tfrac{\ell^{\cap}-1}{2};j\bigr)
\end{equation*}
pairs $(A,B)$ such that
\begin{equation*}
|\{1,t\}\cap A|=|\{1,t\}\cap B|=0\; ;
\end{equation*}

\item[\rm(ix)]
\begin{equation*}
\mathfrak{F}\left(\theta,\beta;\tfrac{\ell+1}{2},\tfrac{\ell'-1}{2},\tfrac{\ell''+1}{2},\tfrac{\ell^{\cap}-1}{2}\right)
\cdot \mathtt{c}\bigl(\tfrac{\ell+1}{2};t-(j'+j''-j)\bigr)\cdot \mathtt{c}\bigl(\tfrac{\ell'-1}{2};j'-j\bigr)
\cdot \mathtt{c}\bigl(\tfrac{\ell''+1}{2};j''-j\bigr)\cdot \mathtt{c}\bigl(\tfrac{\ell^{\cap}-1}{2};j\bigr)
\end{equation*}
pairs $(A,B)$ such that
\begin{equation*}
|\{1,t\}\cap A|=0\ \ \ \text{and}\ \ \ \{1,t\}\cap B=\{t\}\; ;
\end{equation*}

\item[\rm(x)]
\begin{equation*}
\mathfrak{F}\left(\theta,\gamma;\tfrac{\ell+1}{2},\tfrac{\ell'-1}{2},\tfrac{\ell''-1}{2},\tfrac{\ell^{\cap}+1}{2}\right)
\cdot \mathtt{c}\bigl(\tfrac{\ell+1}{2};t-(j'+j''-j)\bigr)\cdot \mathtt{c}\bigl(\tfrac{\ell'-1}{2};j'-j\bigr)
\cdot \mathtt{c}\bigl(\tfrac{\ell''-1}{2};j''-j\bigr)\cdot \mathtt{c}\bigl(\tfrac{\ell^{\cap}+1}{2};j\bigr)
\end{equation*}
pairs $(A,B)$ such that
\begin{equation*}
\{1,t\}\cap A=\{1,t\}\cap B=\{t\}\; .
\end{equation*}
\end{itemize}
\end{theorem}

\begin{proof}
For each pair~$(A,B)$ in the family~(\ref{eq:8}), pick an {\em ascending system\/} of {\em distinct
representatives\/}~$(e_1<e_2<\ldots<e_{\varrho(E_t-(A\cup B))+\varrho(A-B)+\varrho(B-A)+\varrho(A\cap B)})$ \mbox{$\subseteq E_t$} of the intervals composing the sets~$(E_t-(A\cup B))$, $(A-B)$, $(B-A)$ and~\mbox{$(A\cap B)$}. By making the substitutions
\begin{multline*}
e_i\mapsto
\begin{cases}
\theta\; , & \text{if $e_i\in E_t-(A\cup B)$},\\
\alpha\; , & \text{if $e_i\in A-B$},\\
\beta\; , & \text{if $e_i\in B-A$},\\
\gamma\; , & \text{if $e_i\in A\cap B$},
\end{cases}
\\ 1\leq i\leq\varrho(E_t-(A\cup B))+\varrho(A-B)+\varrho(B-A)+\varrho(A\cap B)\; ,
\end{multline*}
we obtain some Smirnov words over the four-letter alphabet~$(\theta,\alpha,\beta,\gamma)$. Given any such Smirnov word, we find the number of pairs~$(A,B)$ in the corresponding subfamily of the family~{\rm(\ref{eq:8})} by means of a product of the form~(\ref{eq:22}).

\noindent(i) Since $\{1,t\}\cap A =\{1,t\}\cap B=\{1\}$,
we have
\begin{equation*}
\ \ \ \varrho(A-B)=\tfrac{\ell'-1}{2}\; ,\ \ \ \varrho(B-A)=\tfrac{\ell''-1}{2}\; ,
\end{equation*}
by \cite[Prop.~2.4(iii)]{M-SC-II}, and
\begin{equation*}
\varrho(A\cap B)=\tfrac{\ell^{\cap}+1}{2}\; ,\ \ \ \varrho(A\cup B)=\tfrac{\ell+1}{2}\; ,
\end{equation*}
by \cite[Prop.~2.4(i)]{M-SC-II}. Note that $\varrho(E_t-(A\cup B))=\varrho(A\cup B)$, that is,
\begin{equation*}
\varrho\bigl(E_t-(A\cup B)\bigr)=\tfrac{\ell+1}{2}\; .
\end{equation*}

\noindent(ii) Since $\{1,t\}\cap A=\{1\}$, and $\{1,t\}\cap B=\{1,t\}$,
we have
\begin{equation*}
\ \ \ \varrho(A-B)=\tfrac{\ell'-1}{2}\; ,\ \ \ \varrho(B-A)=\tfrac{\ell''+1}{2}\; ,
\end{equation*}
by \cite[Prop.~2.4(iii)(iv)]{M-SC-II}, and
\begin{equation*}
\varrho(A\cap B)=\tfrac{\ell^{\cap}+1}{2}\; ,\ \ \ \varrho(A\cup B)=\tfrac{\ell+1}{2}\; ,
\end{equation*}
by \cite[Prop.~2.4(i)(ii)]{M-SC-II}. Note that $\varrho(E_t-(A\cup B))=\varrho(A\cup B)-1$, that is,
\begin{equation*}
\varrho\bigl(E_t-(A\cup B)\bigr)=\tfrac{\ell-1}{2}\; .
\end{equation*}

\noindent(iii) Since $\{1,t\}\cap A=\{1\}$, and $|\{1,t\}\cap B|=0$,
we have
\begin{equation*}
\ \ \ \varrho(A-B)=\tfrac{\ell'+1}{2}\; ,\ \ \ \varrho(B-A)=\tfrac{\ell''-1}{2}\; ,
\end{equation*}
by \cite[Prop.~2.4(i)(iii)]{M-SC-II}, and
\begin{equation*}
\varrho(A\cap B)=\tfrac{\ell^{\cap}-1}{2}\; ,\ \ \ \varrho(A\cup B)=\tfrac{\ell+1}{2}\; ,
\end{equation*}
by \cite[Prop.~2.4(iii)(i)]{M-SC-II}. Note that
\begin{equation*}
\varrho\bigl(E_t-(A\cup B)\bigr)=\tfrac{\ell+1}{2}\; .
\end{equation*}

\noindent(iv) Since $\{1,t\}\cap A=\{1\}$, and $\{1,t\}\cap B=\{t\}$,
we have
\begin{equation*}
\ \ \ \varrho(A-B)=\tfrac{\ell'+1}{2}\; ,\ \ \ \varrho(B-A)=\tfrac{\ell''+1}{2}\; ,
\end{equation*}
by \cite[Prop.~2.4(i)(iv)]{M-SC-II}, and
\begin{equation*}
\varrho(A\cap B)=\tfrac{\ell^{\cap}-1}{2}\; ,\ \ \ \varrho(A\cup B)=\tfrac{\ell+1}{2}\; ,
\end{equation*}
by \cite[Prop.~2.4(iii)(ii)]{M-SC-II}. Note that
\begin{equation*}
\varrho\bigl(E_t-(A\cup B)\bigr)=\tfrac{\ell-1}{2}\; .
\end{equation*}

\noindent(v) Since $\{1,t\}\cap A=\{1,t\}\cap B=\{1,t\}$,
we have
\begin{equation*}
\ \ \ \varrho(A-B)=\tfrac{\ell'-1}{2}\; ,\ \ \ \varrho(B-A)=\tfrac{\ell''-1}{2}\; ,
\end{equation*}
by \cite[Prop.~2.4(iii)]{M-SC-II}, and
\begin{equation*}
\varrho(A\cap B)=\tfrac{\ell^{\cap}+1}{2}\; ,\ \ \ \varrho(A\cup B)=\tfrac{\ell+1}{2}\; ,
\end{equation*}
by \cite[Prop.~2.4(ii)]{M-SC-II}. Note that
\begin{equation*}
\varrho\bigl(E_t-(A\cup B)\bigr)=\tfrac{\ell-1}{2}\; .
\end{equation*}

\noindent(vi) Since $\{1,t\}\cap A=\{1,t\}$, and $|\{1,t\}\cap B|=0$,
we have
\begin{equation*}
\ \ \ \varrho(A-B)=\tfrac{\ell'+1}{2}\; ,\ \ \ \varrho(B-A)=\tfrac{\ell''-1}{2}\; .
\end{equation*}
by \cite[Prop.~2.4(ii)(iii)]{M-SC-II}, and
\begin{equation*}
\varrho(A\cap B)=\tfrac{\ell^{\cap}-1}{2}\; ,\ \ \ \varrho(A\cup B)=\tfrac{\ell+1}{2}\; ,
\end{equation*}
by \cite[Prop.~2.4(iii)(ii)]{M-SC-II}. Note that
\begin{equation*}
\varrho\bigl(E_t-(A\cup B)\bigr)=\tfrac{\ell-1}{2}\; .
\end{equation*}

\noindent(vii) Since $\{1,t\}\cap A=\{1,t\}$, and $\{1,t\}\cap B=\{t\}$,
we have
\begin{equation*}
\ \ \ \varrho(A-B)=\tfrac{\ell'+1}{2}\; ,\ \ \ \varrho(B-A)=\tfrac{\ell''-1}{2}\; ,
\end{equation*}
by \cite[Prop.~2.4(i)(iii)]{M-SC-II}, and
\begin{equation*}
\varrho(A\cap B)=\tfrac{\ell^{\cap}+1}{2}\; ,\ \ \ \varrho(A\cup B)=\tfrac{\ell+1}{2}\; ,
\end{equation*}
by \cite[Prop.~2.4(iv)(ii)]{M-SC-II}. Note that
\begin{equation*}
\varrho\bigl(E_t-(A\cup B)\bigr)=\tfrac{\ell-1}{2}\; .
\end{equation*}

\noindent(viii) Since $|\{1,t\}\cap A|=|\{1,t\}\cap B|=0$,
we have
\begin{equation*}
\ \ \ \varrho(A-B)=\tfrac{\ell'-1}{2}\; ,\ \ \ \varrho(B-A)=\tfrac{\ell''-1}{2}\; .
\end{equation*}
and
\begin{equation*}
\varrho(A\cap B)=\tfrac{\ell^{\cap}-1}{2}\; ,\ \ \ \varrho(A\cup B)=\tfrac{\ell-1}{2}\; ,
\end{equation*}
by \cite[Prop.~2.4(iii)]{M-SC-II}. Note that $\varrho(E_t-(A\cup B))=\varrho(A\cup B)+1$, that is,
\begin{equation*}
\varrho\bigl(E_t-(A\cup B)\bigr)=\tfrac{\ell+1}{2}\; .
\end{equation*}

\noindent(ix) Since $|\{1,t\}\cap A|=0$, and $\{1,t\}\cap B=\{t\}$,
we have
\begin{equation*}
\ \ \ \varrho(A-B)=\tfrac{\ell'-1}{2}\; ,\ \ \ \varrho(B-A)=\tfrac{\ell''+1}{2}\; ,
\end{equation*}
and
\begin{equation*}
\varrho(A\cap B)=\tfrac{\ell^{\cap}-1}{2}\; ,\ \ \ \varrho(A\cup B)=\tfrac{\ell+1}{2}\; ,
\end{equation*}
by \cite[Prop.~2.4(iii)(iv)]{M-SC-II}. Note that $\varrho(E_t-(A\cup B))=\varrho(A\cup B)$, that is,
\begin{equation*}
\varrho\bigl(E_t-(A\cup B)\bigr)=\tfrac{\ell+1}{2}\; .
\end{equation*}

\noindent(x) Since $\{1,t\}\cap A=\{1,t\}\cap B=\{t\}$, we have
\begin{equation*}
\ \ \ \varrho(A-B)=\tfrac{\ell'-1}{2}\; ,\ \ \ \varrho(B-A)=\tfrac{\ell''-1}{2}\; ,
\end{equation*}
by \cite[Prop.~2.4(iii)]{M-SC-II}, and
\begin{equation*}
\varrho(A\cap B)=\tfrac{\ell^{\cap}+1}{2}\; ,\ \ \ \varrho(A\cup B)=\tfrac{\ell+1}{2}\; ,
\end{equation*}
by \cite[Prop.~2.4(iv)]{M-SC-II}. Note that
\begin{equation*}
\varrho\bigl(E_t-(A\cup B)\bigr)=\tfrac{\ell+1}{2}\; .
\end{equation*}
\end{proof}

\noindent$\bullet$ Let us consider the family
\begin{multline}
\label{eq:10}
\bigl\{(A,B)\in\mathbf{2}^{[t]}\times\mathbf{2}^{[t]}
\colon\ \ |A|=:j'<t\; ,\ \  |B|=:j''<t\; ,\\
0<|A\cap B|=:j\; ,\ \ \max\{j',j''\}<|A\cup B|<t\; ,\\
\mathfrak{q}(A\triangle B)=\ell^{\triangle}\; ,\ \
\mathfrak{q}(A\cap B)=\ell^{\cap}\; ,\ \
\mathfrak{q}(A\cup B)=\ell
\bigr\}
\end{multline}
of ordered two-member {\em intersecting Sperner families\/} that {\em do not cover\/} the set $E_t$, with the properties~(\ref{eq:26})(\ref{eq:27}).
\begin{theorem}
Ordered pairs of sets $(A,B)$ in the family~{\rm(\ref{eq:10})} can be counted with the help of products that include subproducts of the form
\begin{multline}
\label{eq:23}
\mathfrak{T}\left(\mathfrak{s}',\; \mathfrak{s}'';\; \varrho(E_t-(A\cup B)),\; \varrho(A\triangle B),\; \varrho(A \cap B)\right)
\\ \times \mathtt{c}\bigl(\varrho(E_t-(A\cup B));t-(j'+j''-j)\bigr)
\\ \times \mathtt{c}\bigl(\varrho(A\triangle B);j'+j''-2j\bigr)\cdot \mathtt{c}\bigl(\varrho(A \cap B);j\bigr)\; :
\end{multline}

\begin{itemize}
\item[\rm(i)]
In the family~{\rm(\ref{eq:10})} there are
\begin{multline*}
\mathfrak{T}\left(\beta,\theta;\tfrac{\ell+1}{2},\tfrac{\ell^{\triangle}-1}{2},\tfrac{\ell^{\cap}+1}{2}\right)\cdot
\mathtt{c}\bigl(\tfrac{\ell+1}{2};t-(j'+j''-j)\bigr)
\cdot\mathtt{c}\bigl(\tfrac{\ell^{\triangle}-1}{2};j'+j''-2j\bigr)\cdot \mathtt{c}\bigl(\tfrac{\ell^{\cap}+1}{2};j\bigr)
\\ \times\tbinom{j'+j''-2j}{j'-j}
\end{multline*}
pairs $(A,B)$ of sets $A$ and $B$ such that
\begin{equation*}
\{1,t\}\cap A =\{1,t\}\cap B=\{1\}\; .
\end{equation*}
\end{itemize}

In the family~{\rm(\ref{eq:10})} there are
\begin{itemize}
\item[\rm(ii)]
\begin{multline*}
\mathfrak{T}\left(\beta,\alpha;\tfrac{\ell-1}{2},\tfrac{\ell^{\triangle}+1}{2},\tfrac{\ell^{\cap}+1}{2}\right)\cdot
\mathtt{c}\bigl(\tfrac{\ell-1}{2};t-(j'+j''-j)\bigr)
\cdot\mathtt{c}\bigl(\tfrac{\ell^{\triangle}+1}{2};j'+j''-2j\bigr)\cdot \mathtt{c}\bigl(\tfrac{\ell^{\cap}+1}{2};j\bigr)
\\ \times\tbinom{j'+j''-2j-1}{j'-j}
\end{multline*}
pairs $(A,B)$ such that
\begin{equation*}
\{1,t\}\cap A=\{1\}\ \ \ \text{and}\ \ \ \{1,t\}\cap B=\{1,t\}\; ;
\end{equation*}

\item[\rm(iii)]
\begin{multline*}
\mathfrak{T}\left(\alpha,\theta;\tfrac{\ell+1}{2},\tfrac{\ell^{\triangle}+1}{2},\tfrac{\ell^{\cap}-1}{2}\right)\cdot
\mathtt{c}\bigl(\tfrac{\ell+1}{2};t-(j'+j''-j)\bigr)
\cdot\mathtt{c}\bigl(\tfrac{\ell^{\triangle}+1}{2};j'+j''-2j\bigr)\cdot \mathtt{c}\bigl(\tfrac{\ell^{\cap}-1}{2};j\bigr)
\\ \times\tbinom{j'+j''-2j-1}{j''-j}
\end{multline*}
pairs $(A,B)$ such that
\begin{equation*}
\{1,t\}\cap A=\{1\}\ \ \ \text{and}\ \ \ |\{1,t\}\cap B|=0\; ;
\end{equation*}

\item[\rm(iv)]
\begin{multline*}
\mathfrak{T}\left(\alpha,\alpha;\tfrac{\ell-1}{2},\tfrac{\ell^{\triangle}+1}{2},\tfrac{\ell^{\cap}-1}{2}\right)\cdot
\mathtt{c}\bigl(\tfrac{\ell-1}{2};t-(j'+j''-j)\bigr)
\cdot\mathtt{c}\bigl(\tfrac{\ell^{\triangle}+1}{2};j'+j''-2j\bigr)\cdot \mathtt{c}\bigl(\tfrac{\ell^{\cap}-1}{2};j\bigr)
\\ \times\tbinom{j'+j''-2j-2}{j'-j-1}
\end{multline*}
pairs $(A,B)$ such that
\begin{equation*}
\{1,t\}\cap A=\{1\}\ \ \ \text{and}\ \ \ \{1,t\}\cap B=\{t\}\; ;
\end{equation*}

\item[\rm(v)]
\begin{multline*}
\mathfrak{T}\left(\beta,\beta;\tfrac{\ell-1}{2},\tfrac{\ell^{\triangle}-1}{2},\tfrac{\ell^{\cap}+1}{2}\right)\cdot
\mathtt{c}\bigl(\tfrac{\ell-1}{2};t-(j'+j''-j)\bigr)
\cdot\mathtt{c}\bigl(\tfrac{\ell^{\triangle}-1}{2};j'+j''-2j\bigr)\cdot \mathtt{c}\bigl(\tfrac{\ell^{\cap}+1}{2};j\bigr)
\\ \times\tbinom{j'+j''-2j}{j'-j}
\end{multline*}
pairs $(A,B)$ such that
\begin{equation*}
\{1,t\}\cap A=\{1,t\}\cap B=\{1,t\}\; ;
\end{equation*}

\item[\rm(vi)]
\begin{multline*}
\mathfrak{T}\left(\alpha,\alpha;\tfrac{\ell-1}{2},\tfrac{\ell^{\triangle}+1}{2},\tfrac{\ell^{\cap}-1}{2}\right)\cdot
\mathtt{c}\bigl(\tfrac{\ell-1}{2};t-(j'+j''-j)\bigr)
\cdot\mathtt{c}\bigl(\tfrac{\ell^{\triangle}+1}{2};j'+j''-2j\bigr)\cdot \mathtt{c}\bigl(\tfrac{\ell^{\cap}-1}{2};j\bigr)
\\ \times\tbinom{j'+j''-2j-2}{j''-j}
\end{multline*}
pairs $(A,B)$ such that
\begin{equation*}
\{1,t\}\cap A=\{1,t\}\ \ \ \text{and}\ \ \ |\{1,t\}\cap B|=0\; ;
\end{equation*}

\item[\rm(vii)]
\begin{multline*}
\mathfrak{T}\left(\alpha,\beta;\tfrac{\ell-1}{2},\tfrac{\ell^{\triangle}+1}{2},\tfrac{\ell^{\cap}+1}{2}\right)\cdot
\mathtt{c}\bigl(\tfrac{\ell-1}{2};t-(j'+j''-j)\bigr)
\cdot\mathtt{c}\bigl(\tfrac{\ell^{\triangle}+1}{2};j'+j''-2j\bigr)\cdot \mathtt{c}\bigl(\tfrac{\ell^{\cap}+1}{2};j\bigr)
\\ \times\tbinom{j'+j''-2j-1}{j''-j}
\end{multline*}
pairs $(A,B)$ such that
\begin{equation*}
\{1,t\}\cap A=\{1,t\}\ \ \ \text{and}\ \ \ \{1,t\}\cap B=\{t\}\; ;
\end{equation*}

\item[\rm(viii)]
\begin{multline*}
\mathfrak{T}\left(\theta,\theta;\tfrac{\ell+1}{2},\tfrac{\ell^{\triangle}-1}{2},\tfrac{\ell^{\cap}-1}{2}\right)\cdot
\mathtt{c}\bigl(\tfrac{\ell+1}{2};t-(j'+j''-j)\bigr)
\cdot\mathtt{c}\bigl(\tfrac{\ell^{\triangle}-1}{2};j'+j''-2j\bigr)\cdot \mathtt{c}\bigl(\tfrac{\ell^{\cap}-1}{2};j\bigr)
\\ \times\tbinom{j'+j''-2j}{j'-j}
\end{multline*}
pairs $(A,B)$ such that
\begin{equation*}
|\{1,t\}\cap A|=|\{1,t\}\cap B|=0\; ;
\end{equation*}

\item[\rm(ix)]
\begin{multline*}
\mathfrak{T}\left(\theta,\alpha;\tfrac{\ell+1}{2},\tfrac{\ell^{\triangle}+1}{2},\tfrac{\ell^{\cap}-1}{2}\right)\cdot
\mathtt{c}\bigl(\tfrac{\ell+1}{2};t-(j'+j''-j)\bigr)
\cdot\mathtt{c}\bigl(\tfrac{\ell^{\triangle}+1}{2};j'+j''-2j\bigr)\cdot \mathtt{c}\bigl(\tfrac{\ell^{\cap}-1}{2};j\bigr)
\\ \times\tbinom{j'+j''-2j-1}{j'-j}
\end{multline*}
pairs $(A,B)$ such that
\begin{equation*}
|\{1,t\}\cap A|=0\ \ \ \text{and}\ \ \ \{1,t\}\cap B=\{t\}\; ;
\end{equation*}

\item[\rm(x)]
\begin{multline*}
\mathfrak{T}\left(\theta,\beta;\tfrac{\ell+1}{2},\tfrac{\ell^{\triangle}-1}{2},\tfrac{\ell^{\cap}+1}{2}\right)\cdot
\mathtt{c}\bigl(\tfrac{\ell+1}{2};t-(j'+j''-j)\bigr)
\cdot\mathtt{c}\bigl(\tfrac{\ell^{\triangle}-1}{2};j'+j''-2j\bigr)\cdot \mathtt{c}\bigl(\tfrac{\ell^{\cap}+1}{2};j\bigr)
\\ \times\tbinom{j'+j''-2j}{j'-j}
\end{multline*}
pairs $(A,B)$ such that
\begin{equation*}
\{1,t\}\cap A=\{1,t\}\cap B=\{t\}\; .
\end{equation*}
\end{itemize}
\end{theorem}

\begin{proof}
For each pair $(A,B)$ in the family~{\rm(\ref{eq:10})}, pick an {\em ascending system\/} of {\em distinct
representatives}~$(e_1<e_2<\ldots<e_{\varrho(E_t-(A\cup B))+\varrho(A\triangle B)+\varrho(A\cap B)})\subseteq E_t$ of the
intervals composing the
sets~$(E_t-(A\cup B))$, $(A\triangle B)$ and~$(A\cap B)$. By making the substitutions
\begin{equation*}
e_i\mapsto
\begin{cases}
\theta\; , & \text{if $e_i\in E_t-(A\cup B)$},\\
\alpha\; , & \text{if $e_i\in A\triangle B$},\\
\beta\; , & \text{if $e_i\in A\cap B$},
\end{cases}
\ \ \ \ \ 1\leq i\leq\varrho(E_t-(A\cup B))+\varrho(A\triangle B)+\varrho(A\cap B)\; ,
\end{equation*}
we obtain some ternary Smirnov words over the alphabet~$(\theta,\alpha,\beta)$. Given any such Smirnov word, we find the number of pairs~$(A,B)$ in the corresponding subfamily of the family~{\rm(\ref{eq:10})} by means of a product of the form~(\ref{eq:23}).

\noindent(i) Since $\{1,t\}\cap A =\{1,t\}\cap B=\{1\}$,
we have
\begin{equation*}
\ \ \ \varrho(A\triangle B)=\tfrac{\ell^{\triangle}-1}{2}\; ,
\end{equation*}
by \cite[Prop.~2.4(iii)]{M-SC-II}, and
\begin{equation*}
\varrho(A\cap B)=\tfrac{\ell^{\cap}+1}{2}\; ,\ \ \ \varrho(A\cup B)=\tfrac{\ell+1}{2}\; ,
\end{equation*}
by \cite[Prop.~2.4(i)]{M-SC-II}. Note that $\varrho(E_t-(A\cup B))=\varrho(A\cup B)$, that is,
\begin{equation*}
\varrho\bigl(E_t-(A\cup B)\bigr)=\tfrac{\ell+1}{2}\; .
\end{equation*}

\noindent(ii) Since $\{1,t\}\cap A=\{1\}$, and $\{1,t\}\cap B=\{1,t\}$, we have
\begin{equation*}
\ \ \ \varrho(A\triangle B)=\tfrac{\ell^{\triangle}+1}{2}\; ,
\end{equation*}
by \cite[Prop.~2.4(iv)]{M-SC-II}, and
\begin{equation*}
\varrho(A\cap B)=\tfrac{\ell^{\cap}+1}{2}\; ,\ \ \ \varrho(A\cup B)=\tfrac{\ell+1}{2}\; ,
\end{equation*}
by \cite[Prop.~2.4(i)(ii)]{M-SC-II}. Note that $\varrho(E_t-(A\cup B))=\varrho(A\cup B)-1$, that is,
\begin{equation*}
\varrho\bigl(E_t-(A\cup B)\bigr)=\tfrac{\ell-1}{2}\; .
\end{equation*}

\noindent(iii) Since $\{1,t\}\cap A=\{1\}$, and $|\{1,t\}\cap B|=0$,
we have
\begin{equation*}
\ \ \ \varrho(A\triangle B)=\tfrac{\ell^{\triangle}+1}{2}\; ,
\end{equation*}
by \cite[Prop.~2.4(i)]{M-SC-II}, and
\begin{equation*}
\varrho(A\cap B)=\tfrac{\ell^{\cap}-1}{2}\; ,\ \ \ \varrho(A\cup B)=\tfrac{\ell+1}{2}\; ,
\end{equation*}
by \cite[Prop.~2.4(iii)(i)]{M-SC-II}. Note that
\begin{equation*}
\varrho\bigl(E_t-(A\cup B)\bigr)=\tfrac{\ell+1}{2}\; .
\end{equation*}

\noindent(iv) Since $\{1,t\}\cap A=\{1\}$, and $\{1,t\}\cap B=\{t\}$, we have
\begin{equation*}
\ \ \ \varrho(A\triangle B)=\tfrac{\ell^{\triangle}+1}{2}\; ,
\end{equation*}
by \cite[Prop.~2.4(ii)]{M-SC-II}, and
\begin{equation*}
\varrho(A\cap B)=\tfrac{\ell^{\cap}-1}{2}\; ,\ \ \ \varrho(A\cup B)=\tfrac{\ell+1}{2}\; ,
\end{equation*}
by \cite[Prop.~2.4(iii)(ii)]{M-SC-II}. Note that
\begin{equation*}
\varrho\bigl(E_t-(A\cup B)\bigr)=\tfrac{\ell-1}{2}\; .
\end{equation*}

\noindent(v) Since $\{1,t\}\cap A=\{1,t\}\cap B=\{1,t\}$,
we have
\begin{equation*}
\ \ \ \varrho(A\triangle B)=\tfrac{\ell^{\triangle}-1}{2}\; ,
\end{equation*}
by \cite[Prop.~2.4(iii)]{M-SC-II}, and
\begin{equation*}
\varrho(A\cap B)=\tfrac{\ell^{\cap}+1}{2}\; ,\ \ \ \varrho(A\cup B)=\tfrac{\ell+1}{2}\; ,
\end{equation*}
by \cite[Prop.~2.4(ii)]{M-SC-II}. Note that
\begin{equation*}
\varrho\bigl(E_t-(A\cup B)\bigr)=\tfrac{\ell-1}{2}\; .
\end{equation*}

\noindent(vi) Since $\{1,t\}\cap A=\{1,t\}$, and $|\{1,t\}\cap B|=0$,
we have
\begin{equation*}
\ \ \ \varrho(A\triangle B)=\tfrac{\ell^{\triangle}+1}{2}\; ,
\end{equation*}
by \cite[Prop.~2.4(ii)]{M-SC-II}, and
\begin{equation*}
\varrho(A\cap B)=\tfrac{\ell^{\cap}-1}{2}\; ,\ \ \ \varrho(A\cup B)=\tfrac{\ell+1}{2}\; ,
\end{equation*}
by \cite[Prop.~2.4(iii)(ii)]{M-SC-II}. Note that
\begin{equation*}
\varrho\bigl(E_t-(A\cup B)\bigr)=\tfrac{\ell-1}{2}\; .
\end{equation*}

\noindent(vii) Since $\{1,t\}\cap A=\{1,t\}$, and $\{1,t\}\cap B=\{t\}$,
we have
\begin{equation*}
\ \ \ \varrho(A\triangle B)=\tfrac{\ell^{\triangle}+1}{2}\; ,
\end{equation*}
by \cite[Prop.~2.4(i)]{M-SC-II}, and
\begin{equation*}
\varrho(A\cap B)=\tfrac{\ell^{\cap}+1}{2}\; ,\ \ \ \varrho(A\cup B)=\tfrac{\ell+1}{2}\; ,
\end{equation*}
by \cite[Prop.~2.4(iv)(ii)]{M-SC-II}. Note that
\begin{equation*}
\varrho\bigl(E_t-(A\cup B)\bigr)=\tfrac{\ell-1}{2}\; .
\end{equation*}

\noindent(viii) Since $|\{1,t\}\cap A|=|\{1,t\}\cap B|=0$,
we have
\begin{equation*}
\ \ \ \varrho(A\triangle B)=\tfrac{\ell^{\triangle}-1}{2}\; ,
\end{equation*}
and
\begin{equation*}
\varrho(A\cap B)=\tfrac{\ell^{\cap}-1}{2}\; ,\ \ \ \varrho(A\cup B)=\tfrac{\ell-1}{2}\; ,
\end{equation*}
by \cite[Prop.~2.4(iii)]{M-SC-II}. Note that $\varrho(E_t-(A\cup B))=\varrho(A\cup B)+1$, that is,
\begin{equation*}
\varrho\bigl(E_t-(A\cup B)\bigr)=\tfrac{\ell+1}{2}\; .
\end{equation*}

\noindent(ix) Since $|\{1,t\}\cap A|=0$, and $\{1,t\}\cap B=\{t\}$,
we have
\begin{equation*}
\ \ \ \varrho(A\triangle B)=\tfrac{\ell^{\triangle}+1}{2}\; ,
\end{equation*}
by \cite[Prop.~2.4(iv)]{M-SC-II}, and
\begin{equation*}
\varrho(A\cap B)=\tfrac{\ell^{\cap}-1}{2}\; ,\ \ \ \varrho(A\cup B)=\tfrac{\ell+1}{2}\; ,
\end{equation*}
by \cite[Prop.~2.4(iii)(iv)]{M-SC-II}. Note that $\varrho(E_t-(A\cup B))=\varrho(A\cup B)$, that is,
\begin{equation*}
\varrho\bigl(E_t-(A\cup B)\bigr)=\tfrac{\ell+1}{2}\; .
\end{equation*}

\noindent(x) Since $\{1,t\}\cap A=\{1,t\}\cap B=\{t\}$,
we have
\begin{equation*}
\ \ \ \varrho(A\triangle B)=\tfrac{\ell^{\triangle}-1}{2}\; ,
\end{equation*}
by \cite[Prop.~2.4(iii)]{M-SC-II}, and
\begin{equation*}
\varrho(A\cap B)=\tfrac{\ell^{\cap}+1}{2}\; ,\ \ \ \varrho(A\cup B)=\tfrac{\ell+1}{2}\; ,
\end{equation*}
by \cite[Prop.~2.4(iv)]{M-SC-II}. Note that
\begin{equation*}
\varrho\bigl(E_t-(A\cup B)\bigr)=\tfrac{\ell+1}{2}\; .
\end{equation*}
\end{proof}

\vspace{5mm}
\end{document}